\documentclass[leqno,11pt]{amsart}

\numberwithin{equation}{section}
\usepackage[utf8]{inputenc}
\usepackage[english]{babel}

\usepackage{amsmath, amsfonts, amssymb, esint, amsthm, nicefrac}
\usepackage{epsfig, graphicx}
\usepackage{amssymb,mathrsfs,enumerate}
\def\endproof{\hspace*{\fill}\mbox{\ \rule{.1in}{.1in}}\medskip }

\newcommand{\R}{\mathbb{R}}
\newcommand{\dist}{{\rm dist}}
\newcommand{\dd}{{\;\rm d}}
\newcommand{\SO}{{SO}}

\newcommand{\sym}{{\rm sym}\,}

\newcommand{\LL}{{L}}
\newcommand{\WW}{{W}}

\newcommand{\trsp}{\scriptscriptstyle{\mathsf{T}}}
\mathchardef\emptyset="001F

\newtheorem{thm}{Theorem}[section]    
\newtheorem{lemma}[thm]{Lemma}

\newtheorem{cor}[thm]{Corollary}

\theoremstyle{definition}

\begin{document}

\title[Dimension reduction for non-Euclidean elasticity]
{Quantitative immersability of Riemann metrics \\ and the infinite
  hierarchy of prestrained shell models} 
\author{Marta Lewicka}
\address{Marta Lewicka: University of Pittsburgh, Department of Mathematics, 
139 University Place, Pittsburgh, PA 15260}
\email{lewicka@pitt.edu} 

\begin{abstract}
This paper concerns the variational description of prestrained
materials, in the context of dimension reduction for thin films
$\Omega^h=\omega\times (-\frac{h}{2}, \frac{h}{2})$.
Given a Riemann metric $G$ on $\Omega^1$,
we study the question of what is the infimum
of the averaged pointwise deficit of a given immersion from being an
orientation-preserving isometric immersion of $G_{\mid \Omega^h}$ on $\Omega^h$,
over all weakly regular immersions. This deficit is measured by the
non-Euclidean energies $\mathcal{E}^h$, which can be seen as
modifications of the classical nonlinear three-dimensional elasticity.

Building on our previous
results, we complete the scaling analysis of $\mathcal{E}^h$ and the derivation of
$\Gamma$-limits of the scaled energies $h^{-2n}\mathcal{E}^h$, for all
$n\geq 1$. We show the energy quantisation, in the sense that
the even powers $2n$ of $h$ are indeed the only possible ones (all of
them are also attained).
For each $n$, we identify the equivalent conditions for the validity
of the corresponding scaling, in terms of the vanishing of appropriate
Riemann curvatures of $G$ up to certain orders, and in terms of the matched isometry
expansions.  We also establish the asymptotic behaviour of the
minimizing immersions as $h\to 0$.
\end{abstract}

\date{\today} 

\maketitle

\section{Introduction}

In this paper, we propose results that address and relate the following two contexts:
\begin{itemize}
\item[(i)] Quantitative analysis of immersability of Riemann metrics.
\item[(ii)] Dimension reduction in non-Euclidean elasticity of prestrained thin films.
\end{itemize}

\smallskip

It is a well-known fact that a three-dimensional Riemann metric $G$
has a smooth isometric immersion in $\R^3$, if an only if its curvature tensor
$R(G)=\{R_{ab,cd}\}_{a,b,c,d=1\ldots 3}$ vanishes identically. The
smoothness requirement may be replaced by the 
orientation-preservation of a Lipschitz continuous immersion; then the condition $R(G)=0$ automatically
yields smoothness and uniqueness, up to
rigid motions. When $R(G)\neq 0$, one may pose the question of what is the infimum
of the average pointwise deficit from being an
orientation-preserving isometric immersion,
over all, weakly regular, immersions. We study this question on a
family of thin films $\big\{\Omega^h=\omega\times (-\frac{h}{2},
\frac{h}{2})\big\}_{h\to 0}$ around a given two-dimensional midplate $\omega$,
where the said deficit is measured by the energy:
$\mathcal{E}^h(u) = \fint_{\Omega^h}\dist^2((\nabla u)G^{-1/2},
SO(3))$. Our first goal is to determine the possible scalings:
$\inf\mathcal{E}^h\sim h^\beta$, as $h\to 0$,  in terms of powers
$\beta$ of the thickness $h$. We are then interested in 
identifying properties of $G$, that correspond to each scaling range, 
in function of the curvature components and their derivatives. Finally,
we want to predict the asymptotics of the minimizing immersions as $h\to 0$.

Similar questions arise in the context of the so-called prestrained elasticity.
A prestrained elastic body is a three-dimensional object,
modeled in its reference configuration by a domain and a
Riemann metric $G$, which is induced by mechanisms such as growth, plasticity, thermal
expansion etc. The body wants to realize the
distances between its constitutive cell elements, which are  set by
$G$, by deforming its shape. Since this realization is taking place in
the flat three-dimensional
space, it is impossible unless $R(G) = 0$. This condition is precisely
equivalent to having the stored non-Euclidean energy of deformations 
infimize to zero. In the variational description of thin prestrained
films $\Omega^h$, we thus study the nonlinear
energies: $\big\{\mathcal{E}^h(u) = \fint_{\Omega^h}W((\nabla
u)G^{-1/2})\}_{h\to 0}$ and, as above, want to determine the viable
scalings of their infima, their singular limits  as $h\to 0$,  and the asymptotic behaviour of the
three-dimensional minimizing shapes.

\smallskip

In our previous works \cite{LP11, BLS} we analyzed the scenario:
$\inf\mathcal{E}^h\sim h^2$, whereas in \cite{LRR,
  LL} we showed that the next limiting
energy level beyond $h^2$ is: $\inf\mathcal{E}^h\sim h^4$, 
arising when $\{R_{12,ab}\}_{a,b=1\ldots 3}=0$
on $\omega$. Then we observed that the further scaling level is:
$\inf\mathcal{E}^h\sim h^6$ and that it corresponds to $R(G)=0$ on
$\omega$. In the present paper, we complete this analysis and provide the derivation of the
$\Gamma$-limits $\mathcal{I}_{2n}$ to scaled energies $h^{-2n}\mathcal{E}^h$, for all
$n\geq 1$. We prove the previously conjectured energy quantisation so that
$h^{2n}$ are indeed the only possible scalings, all of them attained
(by $G=e^{x_3^n}Id_3$.
The structure of $\{\mathcal{I}_{2n}\}_{n\geq 1}$ should be compared with the hierarchy
of plate models in the classical nonlinear elasticity \cite{FJMhier}, as follows. The
energy $\mathcal{I}_2$ consists of pure bending, quantifying the
curvature under the midplate isometric immersion constraint. This is
a Kirchhoff-like model, relative to the ambient metric $G$. The next energy $\mathcal{I}_4$ consists of
linearised first order bending and second order stretching; this is a
von Karman-like model, augmented by terms carrying the relevant
components of the Riemann tensor $R(G)$. Each higher order energy
$\mathcal{I}_{2n}$ consists of linearised bending augmented by the 
the order-related covariant derivatives of $R(G)$ on the
midplate. This is a linear elasticity-like model, in the present
context valid in the quantized scaling regimes $n\geq 3$,
whereas in the classical case appearing in the regimes $h^\beta$ for all $\beta>4$.

\smallskip

Recently, there has been a sustained interest in studying shape formation driven by internal
prestrain, through the experimental, modelling via formal methods,
numerics, and analytical arguments \cite{RHM, sharon, JM,  ESK1}. 
General results have been derived in the abstract
setting of Riemannian manifolds \cite{KS14, KM14, MS18}. Higher energies
$\inf\mathcal{E}^h\sim h^\beta$ with $\beta\in (0,2)$, than the ones 
analyzed in the present paper may result from the interaction of the metric with boundary 
conditions or external forces, leading to the ``wrinkling-like''
effects. Indeed, our setting pertains to the ``no wrinkling'' regime where
$\beta \geq 2$ and the reduced prestrain metric $G_{2\times 2}$ on
$\omega$,  admits a $W^{2,2}$ isometric immersion in $\R^3$. 
While the systematic description of the singular limits at scalings $\beta<2$ is
not yet available, there exists a variety of studies of emerging
patterns: compression- driven {blistering}  \cite{JS, BCDM, BCDM2},
buckling \cite{Ge1, Ge2, Ge3}, origami patterns \cite{Maggi, Venka},
conical singularities \cite{Olber1, Olber2, Olber3}, or coarsening patterns 
\cite{KoBe, Bella, Tobasco}.
In \cite{lemapa2, lemapa2new, LOP}, derivations similar to the results
of the present paper were carried out under a different assumption on the 
asymptotic behavior of the prestrain (constant in the present paper),
which in particular allowed for the effective
energy scalings $h^\beta$ in non-even regimes of $\beta>2$.
On the frontier of experimental modelling of shape formation, we refer
to \cite{9,10,11, 12, biomi}.

\subsection{The set-up of the problem}

Let $\omega\subset\mathbb R^2$ be an open, bounded, connected set
with Lipschitz boundary. We consider a family of thin hyperelastic
sheets occupying the reference domains: 
$$\Omega^h=\omega\times \Big(-\frac{h}{2},\frac{h}{2}\Big)\subset
\R^3, \qquad 0<h\ll 1.$$
A typical point in $\Omega^h$ is denoted by
$x=(x_1,x_2,x_3)=(x',x_3)$. We often use the unit-thickness plate $\Omega^1$ as the referential rescaling of
each $\Omega^h$ via: $\Omega^h\ni (x', x_3)\mapsto (x', x_3/h)\in\Omega^1$.

\medskip

The films $\Omega^h$ are characterized by the given smooth incompatibility (Riemann metric) tensor: 
$$G\in\mathcal{C}^\infty(\bar\Omega^1, \mathbb{R}^{3\times 3}_{\mathrm{sym, pos}})$$
and we want to study the singular limit behaviour, as $h\to 0$, of the following energy functionals:
\begin{equation}\label{functional}
\mathcal{E}^h(u^h)=\frac{1}{h}\int_{\Omega^h} W\big(\nabla
u^h(x)G(x)^{-1/2}\big)\dd x = \int_{\Omega^1} W\big(\nabla u^h(x',
hx_3) G(x', hx_3)^{-1/2}\big)\dd x,
\end{equation}
defined on vector fields  $u^h\in \WW^{1,2}(\Omega^h, \mathbb{R}^3)$
interpreted as deformations of $\Omega^h$. Above,
$G(x)^{-1/2}$ stands for the inverse of $G(x)$. When
$G = Id_3$, the functionals $\mathcal{E}^h$ are the classical Hookean nonlinear
elastic energies of deformations, with the density $W$
obeying the properties listed below. 

\medskip

In the present general setting, $\mathcal{E}^h(u^h)$ is designed to measure the deviation of $u^h$
from being an (equidimensional) isometric immersion of $G$ on $\Omega^h$. Indeed, by
polar decomposition theorem, $F G^{-1/2}\in SO(3)$ if
and only if $F^{\trsp}F=G$ and $\det F>0$. 
The Borel-regular, homogeneous density $W:\mathbb{R}^{3\times
  3}\rightarrow [0,\infty]$ is thus assumed to satisfy:
\begin{itemize}
\item [(i)] 
$W(RF)=W(F)$ for all $R\in \SO(3)$ and $F\in \mathbb{R}^{3\times 3}$,
\item [(ii)]
$W(F)=0$ for all $F\in \SO(3)$,
\item [(iii)] 
$W(F)\geq C\,{\rm dist}^2\big(F,\SO(3)\big)$ for all $F\in \mathbb{R}^{3\times 3}$, with
some uniform constant $C>0$,
\item [(iv)] 
there exists a neighbourhood $\mathcal U$ of $\SO(3)$ such that $W$ is
finite and  $\mathcal{C}^2$ regular on $\mathcal U$.
\end{itemize}
By a more refined analysis \cite{LP11} one can 
prove the global counterpart of the above pointwise
statement, namely that:
$\inf_{W^{1,2}}\mathcal{E}^h =0$ if an only if all the components of
the Riemann curvature tensor of $G$ vanish identically:
$\{R_{ab,cd}\}_{a,b,c,d=1\ldots 3}=0$ on $\Omega^h$.

\medskip

In this paper, we determine the possible energy scalings: $\inf
\mathcal{E}^h\sim h^\beta$ in the limit of vanishing thickness $h\to
0$, and the corresponding variational limits
($\Gamma$-limits) $\mathcal{I}_\beta$ of $h^{-\beta}\mathcal{E}^h$, in the regime $\beta>4$ that
has not been analyzed before. We thus complete the discussion of weakly
prestrained films, started in our previous works \cite{LP11, BLS,
  LRR, LL} that covered the range $\beta\in [2,4]$. 
The singular limits $\mathcal{I}_\beta$ are
typically given by energies of the form $\mathcal{I} = \|Tensor(y)\|^2_{\mathbb{E}}$
defined on the appropriate set of
limiting deformations/displacements $y$ of the midplate $\omega$.
They quantify the resulting effective curvatures in $Tensor(y)$
relative to $G$ at the level induced by $\beta$, and in the weighted
$L^2$ norm on $\omega$:
\begin{equation}\label{poly_norm} 
\mathbb{E}\doteq \big(L^2(\omega,
\mathbb{R}^{2\times 2}_{\sym}), \|\cdot\|_{\mathcal{Q}_2}\big), \qquad
\|F\|_{\mathcal{Q}_2} = \Big(\int_{\omega}\mathcal{Q}_2(x', F(x')) \dd x'\Big)^{1/2}.
\end{equation}
Above, the quadratic form $\mathcal{Q}_2$ carries the two-dimensional
reduction of the first nonzero term in the Taylor expansion of $W$
close to its energy well $SO(3)$. More precisely, we define:
\begin{equation}\label{Qform}
\begin{split}
& \mathcal{Q}_3(F) = D^2 W(Id_3)(F,F)\\
& \mathcal{Q}_2(x', F_{2\times 2}) = \min\left\{
  \mathcal{Q}_3\big(G(x',0)^{-1/2}\tilde F G(x',0)^{-1/2}\big);
  ~ \tilde F\in\mathbb{R}^{3\times 3} \mbox{ with }\tilde F_{2\times 2} = F_{2\times 2}\right\}.
\end{split}
\end{equation}
The form $\mathcal{Q}_3$ is defined for  all $F\in\mathbb{R}^{3\times 3}$, 
while each $\mathcal{Q}_2(x', \cdot)$ is defined on $F_{2\times
  2}\in\mathbb{R}^{2\times 2}$. Both $\mathcal{Q}_3$ and all $\mathcal{Q}_2$ are nonnegative
definite and depend only on the  symmetric parts of their arguments,
in view of the assumptions on $W$.
The quadratic minimization problem in (\ref{Qform}) has thus a unique
solution among symmetric matrices $\tilde F$, which for each $x'\in
\omega$ is given via the linear function:
\begin{equation}\label{cdef}
F_{2\times 2}\mapsto c(x', F_{2\times
  2})\in \mathbb{R}^3\;\; \mbox{ with:} \quad \mathcal{Q}_2(x', F_{2\times 2}) = \mathcal{Q}_3\Big(
G(x',0)^{-1/2} \big(F_{2\times 2}^* + c\otimes e_3\big) G(x',0)^{-1/2}\Big).
\end{equation}

\subsection{Description of the main results of this paper}\label{sec_begin}
As already pointed out, we will be concerned with the regimes of
curvatures of $G$, yielding the incompatibility rate, quantified by $\inf \mathcal{E}^h$, of order higher than
$h^4$ in the thickness $h$. 
We first recall the following result from \cite{LL}:
\begin{equation}\label{h40}
\lim_{h\to 0}\frac{1}{h^4} \inf\mathcal{E}^h = 0 \quad
\Leftrightarrow \quad R_{ab,cd}(x',0) = 0 \quad \mbox{for all } x'\in
\omega,~~ \mbox{for all }\; a,b,c,d:1\ldots 3.
\end{equation}
The above conditions are further equivalent to existence
of smooth vector fields $y_0, \vec b_1, \vec b_2:\bar\omega\to\mathbb{R}^3$, defined
uniquely up to rigid motions, such that
for the following smooth $\R^{3\times 3}$ matrix fields on $\bar \omega$:
$$B_0=\big[ \partial_1 y_0, ~\partial_2 y_0, ~\vec b_1\big],\qquad 
B_1=\big[ \partial_1\vec b_1, ~\partial_2 \vec b_1, ~\vec b_2\big],$$ 
there holds:
\begin{equation}\label{h4}
\begin{split}
& B_0^{\trsp}B_0 = G(x',0) \quad \mbox{ with  } \quad \det B_0 >0\\
& \mbox{and } \quad (B_0^{\trsp}B_1)_{\sym}=\frac{1}{2}\partial_3 G(x',0) \\
& \mbox{and } \quad  \big((\nabla y_0)^{\trsp}\nabla \vec
b_2\big)_{\sym}  + (\nabla\vec b_1)^{\trsp} \nabla \vec b_1= \frac{1}{2}\partial_{33}G(x',0)_{2\times 2}.
\end{split}
\end{equation}
Note that the last equality above implies that we can uniquely define a
new smooth vector and matrix fields: $\vec
b_3:\bar\omega\to\mathbb{R}^3$ and $B_2=\big[ \partial_1 \vec b_2,
~\partial_2 \vec b_2, ~\vec b_3\big]$, so that: $(B_0^{\trsp}
B_2)_{\sym} + B_1^{\trsp} B_1=  \frac{1}{2}\partial_{33}G(x',0)$. This
condition, together with the first two equalities in (\ref{h4}) is
jointly equivalent to:
\begin{equation}\label{h44}
\Big(\sum_{k=0}^2\frac{x_3^k}{k!}B_k\Big)^{\trsp}\Big(\sum_{k=0}^2\frac{x_3^k}{k!}B_k\Big)=
G(x',x_3) + \mathcal{O}(h^3) \quad \mbox{ on } \Omega^h, \quad\mbox{ as } h\to 0.
\end{equation}
In conclusion, the following three conditions: the two conditions in
(\ref{h40}) and the one in (\ref{h44}), are equivalent. Our first main
result generalizes this statement to all even order powers $2(n+1)$ in the infimum
energy scaling, for any $n\geq 2$. Moreover, these scalings exhaust
all possibilities in the remaining regime: $\inf \mathcal{E}^h\sim h^\beta$ with $\beta>4$:

\begin{thm}\label{teo1}
The following three statements are equivalent, for each fixed integer $n\geq 2$:
\begin{itemize}
\item[(i)] $R_{12,12}(x',0) = R_{12,13}(x',0) = R_{12,23}(x',0)=0$ for
  all $x'\in\omega$, and $\partial_3^{(k)} R_{i3,j3}(x',0) = 0$ for all $x'\in\omega$, 
  all $k=0\ldots n-2$ and all $i,j=1\ldots 2$. 
\item[(ii)] $\inf \mathcal{E}^h\leq Ch^{2(n+1)}$.
\item[(iii)] There exist smooth fields $y_0, \{\vec
  b_k\}_{k=1}^{n+1}:\bar\omega\to\mathbb{R}^3$ such that calling $\big\{B_k
  = \big[ \partial_1 \vec b_k, ~\partial_2 \vec b_k, ~\vec
  b_{k+1}\big]\big\}_{k=1}^n$, in addition to $B_0=\big[ \partial_1 y_0, ~\partial_2 y_0, ~\vec
  b_{1}\big]$ satisfying $\det B_0>0$, we have:
\begin{equation}\label{expB}
\Big(\sum_{k=0}^n\frac{x_3^k}{k!}B_k\Big)^{\trsp}\Big(\sum_{k=0}^n\frac{x_3^k}{k!}B_k\Big)=
G(x',x_3) + \mathcal{O}(h^{n+1}) \quad \mbox{ on } \Omega^h, \quad\mbox{ as } h\to 0,
\end{equation}
or in other words: $\displaystyle \sum_{k=0}^{m}\binom{m}{k} B_k^{\trsp} B_{m-k}
-\partial_3^{(m)}G(x',0) = 0\;$ for all $ m=0\ldots n$, for all $x'\in \omega$.
\end{itemize}
\end{thm}

\medskip

We further prove compactness and the lower bound, at any
of the new viable scaling levels $\inf \mathcal{E}^h\sim h^{2(n+1)}$,
completing thus the analysis done for $n=0$ in \cite{LP11, BLS} and for
$n=1$ in \cite{LRR, LL}:

\begin{thm}\label{teo2}
Fix $n\geq 2$ and assume that any of the equivalent conditions in
Theorem \ref{teo1} holds. Let the sequence of deformations $\{u^h\in
W^{1,2}(\Omega^h, \R^3)\}_{h\to 0}$ satisfy: $\mathcal{E}^h(u^h)\leq Ch^{2(n+1)}$. Then:
\begin{itemize}
\item[(i)] There exists $\bar R^h\in SO(3)$, $c^h\in\R^3$ such that the
displacements $\{V^h\in  W^{1,2}(\omega, \R^3)\}_{h\to 0}$ in:
$$V^h(x') = \frac{1}{h^n}\fint_{-h/2}^{h/2} (\bar
R^h)^{\trsp}\big(u^h(x', x_3) - c^h\big) - \Big(y_0(x') +
\sum_{k=1}^n\frac{x_3^k}{k!}\vec b_k(x')\Big)\dd x_3$$
converge as $h\to 0$, strongly in  $W^{1,2}(\omega, \R^3)$, to the limiting displacement:
\begin{equation}\label{Vy0}
V\in\mathcal{V}_{y_0} = \Big\{V\in W^{2,2}(\omega, \R^3);~ \big((\nabla
y_0)^{\trsp}\nabla V\big)_{\sym} =0 \quad \mbox{a.e. in }\omega\Big\}.
\end{equation}
\item[(ii)] The above condition $V\in\mathcal{V}_{y_0}$ automatically defines
  $\vec p\in W^{1,2}(\omega, \R^3)$ such that: 
$$\big(B_0^{\trsp}\big[\nabla V, ~ \vec p\big]\big)_{\sym} =
0\quad\mbox{ a.e. in } \omega,$$
and then we have: $\displaystyle \liminf_{h\to 0 } \frac{1}{h^{2(n+1)}}\mathcal{E}^h(u^h) \geq
\mathcal{I}_{2(n+1)}(V)$, where:
\begin{equation}\label{I2(n+1)}
\begin{split}
\mathcal{I}_{2(n+1)}(V) = & \frac{1}{24}\cdot \Big\|(\nabla y_0)^{\trsp}\nabla
\vec p + (\nabla V)^{\trsp}\nabla \vec b_1 + \alpha_n
\big[\partial_3^{(n-1)} R_{i3,j3}\big]_{i,j=1\ldots
  2}\Big\|^2_{\mathcal{Q}_2} \\ & + \beta_n \cdot 
\Big\|\mathbb{P}_{\mathcal{S}_{y_0}^{\perp}}
\big(\big[\partial_3^{(n-1)} R_{i3,j3}\big]_{i,j=1\ldots
  2}\big)\Big\|^2_{\mathcal{Q}_2} \\ & + \gamma_n \cdot
\Big\|\mathbb{P}_{\mathcal{S}_{y_0}} \big(\big[\partial_3^{(n-1)} R_{i3,j3}\big]_{i,j=1\ldots 2}\big)\Big\|^2_{\mathcal{Q}_2}.
\end{split}
\end{equation}
Above, $\mathcal{S}_{y_0}$ is the following closed subspace of the
Hilbert space $\mathbb{E}$ in (\ref{poly_norm}):
$$\mathcal{S}_{y_0} = \mathrm{closure}_{\mathbb{E}}\Big\{\big((\nabla
y_0)^{\trsp}\nabla w\big)_{\sym};~ w\in W^{1,2}(\omega, \mathbb{R}^3)\Big\},$$
whereas $\mathbb{P}_{\mathcal{S}_{y_0}}(F)$ and
$\mathbb{P}_{\mathcal{S}_{y_0}^{\perp}}(F)$ denote, respectively, 
the orthogonal projections of $F$ onto the space
$\mathcal{S}_{y_0}$ and its orthogonal complement
$\mathcal{S}_{y_0}^{\perp}$ in $\mathbb{E}$.
The coefficients in (\ref{I2(n+1)}) are: 
\begin{equation}\label{albegam}
\begin{split}
& \alpha_n = \left\{\begin{array}{ll} 0 & \mbox{for }\; n \mbox{ odd}\\
    \displaystyle \frac{3}{2^n (n+3) (n+1)!} & \mbox{for }\; n \mbox{
      even} \end{array}\right. , \\ & \beta_n = 
\frac{1}{2^{2n+3}(2n+3)\big((n+1)!\big)^2}\cdot \left\{\begin{array}{ll} 1 & \mbox{for }\; n \mbox{ odd}\\
    \displaystyle \frac{n^2}{(n+3)^2} & \mbox{for }\; n \mbox{ even} \end{array}\right. , \\ & 
\gamma_n = \frac{1}{2^{2n+3}(2n+3)\big((n+1)!\big)^2}\cdot
\left\{\begin{array}{ll} \displaystyle \frac{(n+1)^2}{(n+2)^2} & \mbox{for }\; n \mbox{ odd}\\
    \displaystyle \frac{n^2}{(n+3)^2} & \mbox{for }\; n \mbox{ even} \end{array}\right. .
\end{split}
\end{equation}
\item[(iii)] There holds on $\omega$:
\begin{equation}\label{riem}
\begin{split}
2\cdot\big[\partial_3^{(n-1)} R_{i3,j3}(\cdot,0)\big]_{i,j=1\ldots 2} = \; &
2\big((\nabla y_0)^{\trsp}\nabla\vec b_{n+1}\big)_{\sym}
+ \sum_{k=1}^{n}\binom{n+1}{k} (\nabla \vec
b_k)^{\trsp}\nabla \vec b_{n+1-k} \\ & - \partial_3^{(n+1)}G(\cdot,0)_{2\times 2}.
\end{split}
\end{equation}
\end{itemize}
\end{thm}

\medskip

We point out a few related observations:
\begin{itemize}
\item[(i)] When $G=Id_3$, then each functional in (\ref{I2(n+1)}) reduces to the classical linear elasticity. 
We have: $y_0=id$, $\vec b_1=e_3$ and 
$ \mathcal{V} = \big\{(\alpha x^\perp +\vec \beta, v); ~
\alpha\in\mathbb{R}, ~ \vec \beta\in\mathbb{R}^2, ~ v\in W^{2,2}
(\omega)\big\}$, and for $V\in \mathcal{V}$, there holds: $\vec p = (-\nabla
v, 0)$. Consequently: $\displaystyle \mathcal{I}_{2(n+1)}(V)= \frac{1}{24}\int_\omega
  \mathcal{Q}_2\big(x', \nabla^2 v\big) \dd x'$,  in function of the out-of-plane scalar displacement $v$.
\item[(ii)] In the present geometric context, the bending term is: $(\nabla y_0)^{\trsp}\nabla
\vec p + (\nabla V)^{\trsp}\nabla \vec b_1$. It is of order $h^{n}x_3$
and it interacts with the curvature term $\big[\partial_3^{(n-1)}
R_{i3,j3}(\cdot,0)\big]_{i,j=1\ldots 2}$, which is of order $x_3^{n+1}$. 
The interaction occurs only when the two terms have same parity, which
happens at even $n$, so $\alpha_n=0$ for $n$ odd. 
The two remaining terms in (\ref{I2(n+1)}) measure the
$L^2$ norm of $\big[\partial_3^{(n-1)} R_{i3,j3}(\cdot,0)\big]_{i,j=1\ldots 2}$, with distinct weights
assigned to  the $\mathcal{S}_{y_0}$ and
$\big(\mathcal{S}_{y_0}\big)^\perp$ projections, again according to the parity of $n$. 
\item[(iii)] The formula in (\ref{riem}) relates the
quantities appearing in conditions (i) and (iii) of Theorem
\ref{teo1}. The curvature $\big[\partial_3^{(n-1)}
R_{i3,j3}(\cdot,0)\big]_{i,j=1\ldots 2}$ is thus precisely the coefficient of the discrepancy of the order
$h^{n+1}$ in (\ref{expB}) at the $2\times 2$ minor, scaled by the $(n+1)!/2$ factor.
\item[(iv)] The finite strain space $\mathcal{S}_{y_0}$ can be
  identtified in the the following two cases.When $y_0=id_2$, then $\mathcal{S}_{y_0}= \{\mathbb{S}\in
L^2(\omega, \mathbb{R}^{2\times 2}_{\sym});
~\mbox{curl}^{\trsp}\mbox{curl}\,\mathbb{S}=0\}$.
When the Gauss curvature $\kappa((\nabla y_0)^{\trsp}\nabla
y_0)=\kappa\big(G_{2\times 2})>0$ in $\bar
\omega$, then $\mathscr{S}_{y_0}= L^2(\omega, \mathbb{R}^{2\times
  2}_{\sym})$, as shown in \cite{lemopa}. 
\end{itemize}

\medskip

Our next result proves the upper bound that is consistent with
Theorem \ref{teo1} and yields the $\Gamma$-convergence of the rescaled energies
$h^{-2(n+1)}\mathcal{E}^h$ to the dimensionally reduced limits
$\mathcal{I}_{2(n+1)}$ in (\ref{I2(n+1)}):

\begin{thm}\label{teo3}
Fix $n\geq 2$ and assume that any of the equivalent conditions in
Theorem \ref{teo1} hold. Then for every $V\in \mathcal{V}_{y_0}$ as
defined in (\ref{Vy0}), there exists a sequence $\{u^h\in
W^{1,2}(\Omega^h, \R^3)\}_{h\to 0}$ so that:
\begin{equation}\label{goodappro}
\frac{1}{h^n}\fint_{-h/2}^{h/2} u^h(x', x_3) - \Big(y_0 +
\sum_{k=1}^n\frac{x_3^k}{k!}\vec b_k\Big)\dd x_3 \to V \quad \mbox{as
} h\to 0, \quad \mbox{strongly in } \; W^{1,2}(\omega, \R^3),
\end{equation}
and that: $\displaystyle \liminf_{h\to 0 }
\frac{1}{h^{2(n+1)}}\mathcal{E}^h(u^h) = \mathcal{I}_{2(n+1)}(V),$
where the limiting energy functional is as in (\ref{I2(n+1)}).
\end{thm}

\medskip

It is worth noting the following self-evident application of Theorems
\ref{teo1}, \ref{teo2} and \ref{teo3}:

\begin{cor}\label{cor4}
Under either of the equivalent conditions in (\ref{h40}), assume
that for some $n\geq 2$ there holds: $\partial_3^{(m)}
\big[R_{i3j3}(\cdot ,0)\big]_{i,j=1\ldots 2} = 0$  on $\omega$, for  all $m=0\ldots n-2$, but 
$\partial_3^{(n-1)} \big[R_{i3j3}(\cdot ,0)\big]_{i,j=1\ldots 2}
\not\equiv 0$. Then there exist $c, C>0$ such that:
\begin{equation}\label{allscal}
ch^{2(n+1)}\leq \inf \mathcal{E}^h\leq Ch^{2(n+1)}.
\end{equation}
Moreover, the scaled energies $\displaystyle \frac{1}{h^{2(n-1)}}\mathcal{E}^h$,
$\Gamma$-converge to the limiting functional $\mathcal{I}_{2(n+1)}$ in  (\ref{I2(n+1)}),
effectively defined on the space $\mathcal{V}_{y_0}$  of first order infinitesimal
isometries in (\ref{Vy0}). 
\end{cor}

\medskip

For completeness, we note that the conformal metrics of the form: $G(x', x_3) =
e^{2\phi(x_3)}Id_3$ provide a class of examples for the viability of
all scalings in (\ref{allscal}). Indeed, the trace midplate metric $e^{2\phi(0)}Id_2$ has
a smooth isometric immersion $y_0=e^{\phi(0)}id_2:\omega\to\R^2$, and 
the only possibly nonzero Riemann curvatures of $G$ are given by:
$ R_{1212}=-\phi'(x_3)^2e^{2\phi(x_3)}$, 
$R_{1313}=R_{2323}=-\phi''(x_3)e^{2\phi(x_3)}$. By
Corollary \ref{cor4} we see that 
$\inf \mathcal{E}^h\sim h^{2n}$ if and only if $\phi^{(k)}(0)=0$ for
$k=1\ldots n-1$ and $\phi^{(n)}(0)\neq 0$. 

\subsection{The structure of the paper} In sections \ref{2} and \ref{sec_geometry}
we work under the assumption (iii) of Theorem \ref{teo1}. First, in Lemma \ref{3to2}, we give an 
easy proof of the implication $(iii)\Rightarrow (ii)$. The particular
energy-consistent deformation field can be further used as the local
change of variables allowing for the application of the nonlinear
rigidity estimate \cite{FJM02} in the present context. This is done
in Lemma \ref{approx1} and Corollary \ref{approx2}, providing an
approximation of an arbitrary energy-consistent deformation gradient
$\nabla u^h$, by a non-symmetric square root
of the $n$-th order Taylor expansion of the metric $G$, derived from the expansion
guaranteed in (iii). Both the approximation error and the $L^2$ norm
of the gradient of the rotation field are energy-controlled. In Lemma \ref{compact}
we prove the compactness part of Theorem \ref{teo2}. In Lemma
\ref{lowbd} we conclude a preliminary lower bound estimate, involving
a version of the functional $\mathcal{I}_{2(n+1)}$, whose curvature terms
are still expressed in terms of the expansion fields in (iii), as suggested in
the right hand side of (\ref{riem}). 

In section \ref{sec_geometry}, we develop a geometric line of arguments, serving to prove
(in Corollary \ref{t1iii}) the identity (\ref{riem}) under assumption
(iii). In Lemmas \ref{n1} and \ref{n2}, we partially reprove
the equivalent conditions valid at the previously analyzed scalings $h^2$
and $h^4$. These statements are then generalized in Lemma
\ref{industep}, where we show the implication (iii) $\Rightarrow$ (i),
resulting also in the existence of a one order higher approximate field
$\vec b_{n+1}$, that is given solely through the Christoffel symbols of $G$ on $\omega$. 

In section \ref{sec_I2(n+1)} we finally
prove Theorem \ref{teo1}, showing equivalence of
the stated three conditions, by induction on $n\geq 2$. We also finish the
proof of Theorem \ref{teo2} by: improving the lower bound
from section \ref{2}, identifying its curvature components via (\ref{riem}),
and separating the bending and the excess terms.
In section \ref{sec_recseq} we prove Theorem \ref{teo3}, constructing
a energy-consistent recovery sequence.

\subsection{Notation} 
Given a matrix $F\in\mathbb{R}^{n\times n}$, we denote 
its transpose by $F^{\trsp}$ and its symmetric part by $F_{\mathrm{sym}} =
\frac{1}{2}(F + F^{\trsp})$. The space of symmetric $n\times n$ matrices is
denoted by $\R^{n\times n}_{\mathrm{sym}}$, whereas $\R^{n\times
  n}_{\mathrm{sym, pos}}$ stands for the space of symmetric, positive
definite  $n\times n$ matrices.
By $SO(n) = \{R\in\mathbb{R}^{n\times n}; ~ R^{\trsp} = R^{-1} \mbox{ and }
  \det R=1\}$ we mean the group of special rotations; its tangent
  space at $Id_n$ consists of skew-symmetric matrices:  $T_{Id_n}SO(n)=so(n) =
\{F\in \mathbb{R}^{n\times n}; ~ F_{\mathrm{sym}}=0\}$.
We use the matrix norm $|F| = (\mbox{trace}(F^{\trsp} F))^{1/2}$, which
is induced by the inner product $\langle F_1 : F_2\rangle =
\mbox{trace}(F_1^{\trsp} F_2)$. The $2 \times 2$ principal minor of 
$F\in\mathbb{R}^{3\times 3}$ is denoted by $F_{2\times 2}$.
Conversely, for a given $F_{2\times 2} \in \mathbb{R}^{2\times 2}$, the $3 \times 3$ matrix with
principal minor equal $F_{2\times 2}$ and all other entries equal to
$0$, is denoted by $F_{2\times 2}^*$. Unless specified otherwise, all limits are
taken as the thickness parameter $h$ vanishes: $h\to 0$. By $C$ we denote any
universal positive constant, independent of $h$. 

\subsection{Acknowledgments}
M.L. is grateful to Annie Raoult and Shankar Venkataramani for
interest in the project and discussions. Support by the NSF grant DMS-1613153 is acknowledged.

\section{A proof of Theorem \ref{teo2}: compactness and a preliminary lower bound}\label{2}

In this section, assuming condition (iii) of Theorem \ref{teo1}, we
derive the compactness and (a version of) the lower bound in Theorem \ref{teo2}.
We first observe the implication $(iii)\Rightarrow (ii)$ in
Theorem \ref{teo1}:

\begin{lemma} \label{3to2}
Assume that condition (iii) in Theorem \ref{teo1} holds, for some
$n\geq 2$. Then we have: $$\inf\mathcal{E}^h\leq C h^{2(n+1)}.$$
\end{lemma}
\begin{proof}
Define $\displaystyle u^h(x',x_3) = y_0 + \sum_{k=1}^{n+1}\frac{x_3^k}{k!}\vec b_k$, so that:
$$\nabla u^h(x', x_3) = \sum_{k=0}^n\frac{x_3^k}{k!}B_k +
\frac{x_3^{n+1}}{(n+1)!}\big[ \partial_1 \vec b_{n+1},~ \partial_2 \vec b_{n+1},~0\big].$$
Consequently, $(\nabla u^h) G^{-1/2}$ is positive definite for all
small $h$, and modulo a rotation field it equals the following matrix
field on $\Omega^h$, where we used the assumption (\ref{expB}):
\begin{equation*} 
\begin{split}
\sqrt{\big((\nabla u^h) G^{-1/2}\big)^{\trsp}\big((\nabla u^h)
  G^{-1/2}\big)} & = \sqrt{Id_3 + G^{-1/2}\big((\nabla u^h)^{\trsp}\nabla u^h-G\big)
  G^{-1/2}} \\ & = \sqrt{Id_3 + \mathcal{O}(h^{n+1})} = Id_3 + \mathcal{O}(h^{n+1}).
\end{split}
\end{equation*}
This implies: $\displaystyle \mathcal{E}^h(u^h) = \frac{1}{h} \int_{\Omega^h} W\big(Id_3 +
\mathcal{O}(h^{n+1})\big)\dd x\leq C h^{2(n+1)}$, as claimed.
\end{proof}

Recalling results (\ref{h40}) and (\ref{h44}) quoted from \cite{LL}, we already
see that $\displaystyle{\lim_{h\to
    0}\frac{1}{h^4}\inf\mathcal{E}^h}= 0$ automatically implies:
$\inf\mathcal E^h \leq Ch^6$. Before addressing compactness at
$h^{2n}$ with $h\geq 3$, we develop the nonlinear rigidity estimates
applicable in the present context.

\begin{lemma} \label{approx1}
Assume that condition (iii) in Theorem \ref{teo1} holds, for some
$n\geq 2$. Let $V\subset\omega$ be an open, Lipschitz subdomain such
that $y_0$ is injective on $V$. Denote $V^h=V\times
(-\frac{h}{2},\frac{h}{2})$. Then for every $u^h\in W^{1,2}(V^h,
\R^3)$ there exists $\bar R^h\in SO(3)$ such that:
$$\frac{1}{h}\int_{V^h}\big|\nabla u^h - \bar
R^h\sum_{k=0}^n\frac{x_3^k}{k!}B_k\big|^2\dd x \leq
C\Big(\frac{1}{h}\int_{V^h}W\big((\nabla u^h)G^{-1/2}\big)\dd x + h^{2n+1}|V^h|\Big).$$
The constant $C$ is uniform for all $V^h\subset\Omega^1$ that are
bi-Lipschitz equivalent with controlled Lipschitz constants.
\end{lemma}
\begin{proof}
Define $\displaystyle Y = y_0 + \sum_{k=1}^{n+1}\frac{x_3^k}{k!}\vec
b_k$, and observe that for $h$ sufficiently small, $Y$ is a smooth
diffeomorphism of $V^h$ onto its image $U^h\subset \R^3$. Consider the
change of variables $v^h=u^h\circ Y^{-1}\in W^{1,2}(U^h, \R^3)$ and
apply the fundamental geometric rigidity estimate \cite{FJM02}, yielding
existence of $\bar R^h\in SO(3)$ with:
$$\int_{U^h}|\nabla v^h - \bar R^h|^2 \leq C
\int_{U^h}\dist^2\big(\nabla v^h, SO(3)\big).$$
Changing variable in the left hand side gives:
\begin{equation*}
\begin{split}
\int_{U^h}|\nabla v^h - \bar R^h|^2 & =\int_{V^h}(\det \nabla
Y)\cdot \big|(\nabla u^h)(\nabla Y)^{-1} - \bar R^h\big|^2 \geq 
C \int_{V^h}\big|\nabla u^h - \bar R^h \nabla Y\big|^2 \\ & = 
C \int_{V^h}\big|\nabla u^h - \bar R^h \big(\sum_{k=0}^n\frac{x_3^k}{k!}B_k\big)\big|^2 +
C\int_{V^h}\mathcal{O}(h^{2(n+1)}).
\end{split}
\end{equation*}
Changing now variable in the right hand side and using 
$(\nabla Y)G^{-1/2}\in SO(3)\big(Id_3 + \mathcal{O}(h^{n+1})\big)$,
as established in Lemma \ref{3to2}, results in:
\begin{equation*}
\begin{split}
\int_{U^h}\dist^2\big(\nabla v^h, SO(3)\big) & =\int_{V^h}(\det \nabla
Y)\cdot \dist^2\big((\nabla u^h)(\nabla Y)^{-1}, SO(3)\big) \\ & \leq 
C \int_{V^h}\dist^2\big((\nabla v^h)G^{-1/2}, SO(3)(\nabla Y)G^{-1/2}\big)
\\ & \leq C \int_{V^h}\dist^2\big((\nabla v^h)G^{-1/2}, SO(3)\big) + 
C\int_{V^h}\mathcal{O}(h^{2(n+1)}).
\end{split}
\end{equation*}
Combining the three displayed inequalities above proves the result.
\end{proof}

The well-known approximation technique \cite{FJMhier} together with
the arguments in \cite[Corollary 2.3]{LRR}, yield the following 
estimate, whose proof we leave to the reader:
\begin{cor}\label{approx2}
Assume condition (iii) in Theorem \ref{teo1}, for some
$n\geq 2$. Then, given a sequence $\{u^h\in W^{1,2}(\Omega^h,
\R^3)\}_{h\to 0}$ such that $\mathcal{E}^h(u^h)\leq Ch^{2(n+1)}$, there
exists $\{R^h\in W^{1,2}(\omega, SO(3))\}_{h\to 0}$ with:
\begin{equation*}
\begin{split}
\frac{1}{h}\int_{\Omega^h} \big|\nabla u^h -
R^h\sum_{k=0}^n\frac{x_3^k}{k!}B_k\big|^2\dd x\leq Ch^{2(n+1)}\quad
\mbox{ and }\quad
\int_\omega |\nabla R^h(x')|^2\dd x' \leq C h^{2n}.
\end{split}
\end{equation*}
\end{cor}

\medskip

We now show the compactness part of Theorem \ref{teo1}:

\begin{lemma}\label{compact}
Assume condition (iii) in Theorem \ref{teo1}, for some $n\geq 2$.
Let the sequence of deformations $\{u^h\in
W^{1,2}(\Omega^h, \R^3)\}_{h\to 0}$ satisfy: $\mathcal{E}^h(u^h)\leq
Ch^{2(n+1)}$. Then:
\begin{itemize}
\item[(i)] The averaged displacements $V^h$ converge, up to a subsequence, 
to the first order isometry $V$ as in Theorem \ref{teo2} (i).
\item[(ii)] The scaled strains $\displaystyle \frac{1}{h}\big((\nabla
  y_0)^{\trsp}\nabla V^h\big)_{\sym}$ converge, up to a subsequence,
  weakly in $L^2(\omega, \R^{3\times 3})$ to some $\mathbb{S}\in \mathcal{S}_{y_0}$.
\end{itemize}
\end{lemma}
\begin{proof}
{\bf 1.} Define the following rotation: $\displaystyle \bar R^h = \mathbb{P}_{SO(3)}\fint_{\Omega^h}
(\nabla u^h)\Big(\sum_{k=0}^n\frac{x_3^k}{k!} B_k\Big)^{-1}\dd x$. In
order to observe that the above definition is legitimate, we write:
\begin{equation*}
\begin{split}
\dist^2\Big(\fint_{\Omega^h} (\nabla u^h)&\Big(\sum_{k=0}^n\frac{x_3^k}{k!} B_k\Big)^{-1}\dd x,
SO(3)\Big)  \leq \Big|\fint_{\Omega^h}
(\nabla u^h)\Big(\sum_{k=0}^n\frac{x_3^k}{k!} B_k\Big)^{-1}\dd x -
R^h(x')\Big|^2 \\ & \leq 2 \fint_{\Omega^h}
\big|(\nabla u^h)\Big(\sum_{k=0}^n\frac{x_3^k}{k!} B_k\Big)^{-1} -
R^h\big|^2\dd x +  2 \big|\big(\fint_\omega R^h\dd x'\big) - R^h(x')\big|^2,
\end{split}
\end{equation*}
and upon integrating $\dd x'$ on the domain $\omega$ while noting Corollary \ref{approx2}, obtain:
\begin{equation*}
\dist^2\Big(\fint_{\Omega^h} (\nabla u^h)\Big(\sum_{k=0}^n\frac{x_3^k}{k!} B_k\Big)^{-1}\dd x,
SO(3)\Big)  \leq Ch^{2(n+1)} + C h^{2n}\leq Ch^{2n}.
\end{equation*}
Consequently, there also follows:
\begin{equation}\label{use1}
\Big|\fint_{\Omega^h} (\nabla u^h)\Big(\sum_{k=0}^n\frac{x_3^k}{k!}
B_k\Big)^{-1}\dd x - \bar R^h\Big|^2\leq Ch^{2n}, \qquad \fint_{\omega}
|R^h - \bar R^h|^2\dd x' \leq Ch^{2n}.
\end{equation}
Set now $c^h\in \R^3$ so that $\int_\omega V^h\dd x'= 0$. We get:
\begin{equation}\label{nabVh}
\begin{split}
\nabla V^h = & \;\frac{1}{h^n}\fint_{-h/2}^{h/2} (\bar R^h)^{\trsp}\big[\partial_1 u^h, \partial_2 u^h\big] - 
(\bar R^h)^{\trsp} R^h\Big(\sum_{k=0}^n\frac{x_3^k}{k!}B_k\Big)_{3\times 2}\dd x_3
\\ & + S^h \fint_{-h/2}^{h/2} \Big(\sum_{k=0}^n\frac{x_3^k}{k!}B_k\Big)_{3\times 2}\dd x_3,
\end{split}
\end{equation}
where we define the following matrix fields whose convergence (up to
a subsequence) results
from the second bound in (\ref{use1}) and from Corollary \ref{approx2}:
\begin{equation}\label{Sh}
S^h = \frac{1}{h^n}\Big((\bar R^h)^{\trsp}R^h - Id_3\Big)
\rightharpoonup S \quad \mbox{ weakly in }\; W^{1,2}(\omega, \R^{3\times 3}).
\end{equation}
We also note that $S\in so(3)$ a.e. in $\omega$.  Since the first
term in the right hand side of (\ref{nabVh}) converges to $0$ in
$L^2(\omega)$, in virtue of Corollary \ref{approx2}, we conclude the
following convergence, up to a subsequence:
$$\nabla V^h\to (SB_0)_{3\times 2} = S \nabla y_0 \quad \mbox{
  strongly in } \; L^2(\omega,\mathbb{R}^{3\times 2}).$$
It also follows that the limit $S\nabla y_0 \in W^{1,2}(\omega,
\R^{3\times 2})$. A further application of the Poincare inequality to
the mean-zero displacements $V^h$, yields their strong convergence (up to
a subsequence in $W^{1,2}(\omega,\R^3)$) to some $V\in
W^{2,2}(\omega,\R^3)$ satisfying $\nabla V = (SB_0)_{3\times 2}$. By
skew-symmetry of $S$, it follows that $(\nabla y_0)^{\trsp}\nabla V$
is skew a.e. in $\omega$, proving (i).

\medskip

{\bf 2.} We observe that the first term in the right hand side of
(\ref{nabVh}) has its $L^2(\omega)$ norm bounded by $Ch^2$, in view of the
first estimate in Corollary \ref{approx2}. Consequently, in the
decomposition of $\displaystyle \frac{1}{h}\big((\nabla
y_0)^{\trsp}\nabla V^h\big)_{\sym}$, parallel to that in (\ref{nabVh}),
the corresponding first term has a weakly converging subsequence. 
The remaining second term equals:
$$\frac{1}{h}\Big((\nabla y_0)^{\trsp} S^h \big(\nabla y_0 +
\mathcal{O}(h^2)\big)\Big)_{\sym} = \frac{1}{h}(\nabla y_0)^{\trsp} S^h_{\sym} \nabla y_0
+ \mathcal{O}(h|S^h|).$$
The $L^2(\omega)$ norm of the second term above clearly converges to $0$,
whereas the first term obeys: 
\begin{equation}\label{Ssym}
\frac{1}{h}S^h_{\sym}= -\frac{h^{n-1}}{2} (S^h)^{\trsp}S^h\to 0
\quad \mbox{ strongly in }\; L^2(\omega, \R^{3\times 3}).
\end{equation}
This ends the proof of the claim.
\end{proof}

\medskip

We are now ready to derive the lower bound on the scaled energies
$h^{-2(n+1)}\mathcal{E}^h(u^h)$, in terms of the expansion fields
$y_0, \{\vec b_k\}_{k=1}^{n+1}$ in condition (iii) of Theorem \ref{teo1}:

\begin{lemma}\label{lowbd}
In the context of Lemma \ref{compact}, there holds:
\begin{equation*}
\begin{split}
&\liminf_{h\to 0}  \frac{1}{h^{2(n+1)}}\mathcal{E}^h(u^h) \\ & \quad \geq  \frac{1}{2} \int_{\Omega^1}\mathcal{Q}_2\bigg(x',~
\mathbb{S} - \delta_{n+1} \big((\nabla y_0)^{\trsp}\nabla \vec b_{n+1}\big)_{\sym} 
+ x_3 \big((\nabla y_0)^{\trsp}\nabla \vec p + (\nabla V)^{\trsp}\nabla \vec b_1\big)
\\ & \qquad\quad
+ \frac{x_3^{n+1}}{2(n+1)!} \Big( 2\big((\nabla y_0)^{\trsp}\nabla \vec
b_{n+1}\big)_{\sym} + \sum_{k=1}^n\binom{n+1}{k}
(\nabla \vec b_k)^{\trsp}\nabla\vec b_{n+1-k} - \partial_{3}^{(n+1)}G(x',0)\Big)  \bigg)\dd x,
\end{split}
\end{equation*}
with the coefficient $\delta_{n+1}$ given by:
\begin{equation}\label{deldef}
\delta_{n+1} = \left\{\begin{array}{ll} \displaystyle
    \frac{1}{(n+2)!2^{n+1}} &\mbox{ for } \; n \mbox{ odd} \\
0 &\mbox{ for } \; n \mbox{ even} \end{array}\right. . 
\end{equation}
\end{lemma}
\begin{proof}
{\bf 1.} By Corollary \ref{approx2}, the following matrix fields
$\{\mathcal{Z}^h\in L^2(\Omega^1, \R^{3\times 3})\}_{h\to 0}$ have a
converging subsequence, weakly in $L^2(\Omega^1, \R^{3\times 3})$:
\begin{equation}\label{Zh}
\mathcal{Z}^h(x', x_3) = \frac{1}{h^{n+1}}\Big(\nabla u^h(x', hx_3) -
R^h(x')\sum_{k=0}^n\frac{h^kx_3^k}{k!}B_k(x')\Big) \rightharpoonup \mathcal{Z}.
\end{equation}
We write: $\displaystyle (R^h)^{\trsp}\nabla
u^h(x', hx_3) = \sum_{k=0}^n\frac{h^kx_3^k}{k!}B_k + h^{n+1}(R^h)^{\trsp}\mathcal{Z}^h$
and observe that:
\begin{equation}\label{comp}
\begin{split}
\mathcal{E}^h(u^h) & =\int_{\Omega^1} W\big((R^h)^{\trsp}\nabla
u^h(x', hx_3) G(x', hx_3)^{-1/2}\big)\dd x \\ & \geq
\int_{\{|\mathcal{Z}^h|^2\leq \frac{1}{h}\}} W\big(\sqrt{Id_3 + G(x',
  hx_3)^{-1/2}J^h G(x', hx_3)^{-1/2}}\big)\dd x,
\end{split}
\end{equation}
where the intermediary field $J^h$ has the following expansion, on the
set $\{|\mathcal{Z}^h|^2\leq \frac{1}{h}\}\subset\Omega^1$:
\begin{equation*}
\begin{split}
J^h(x', x_3)  = & \;
\Big(\sum_{k=0}^n\frac{h^kx_3^k}{k!}B_k\Big)^{\trsp}
\Big(\sum_{k=0}^n\frac{h^kx_3^k}{k!}B_k\Big) - G(x', hx_3) \\ & +
2h^{n+1}\bigg(\Big(\sum_{k=0}^n\frac{h^kx_3^k}{k!}B_k\Big)^{\trsp} (R^h)^{\trsp}\mathcal{Z}^h\bigg)_{\sym}
+ h^{2(n+1)}(\mathcal{Z}^h)^{\trsp}\mathcal{Z}^h \\  = & \;
\frac{h^{n+1}x_3^{n+1}}{(n+1)!} \Big(\sum_{k=1}^n\binom{n+1}{k}
(B_k)^{\trsp}B_{n+1-k} - \partial_{3}^{(n+1)}G(x',0)\Big) \\ & +
2h^{n+1}\big(B_0^{\trsp} (R^h)^{\trsp}\mathcal{Z}^h\big)_{\sym} + o(h^{n+1})
\end{split}
\end{equation*}
Consequently, we get from (\ref{comp}) and Taylor expanding $W$ at $Id_3$:
\begin{equation*}
\begin{split}
\frac{1}{h^{2(n+1)}}\mathcal{E}^h(u^h)  \geq \frac{1}{2}
\int_{\{|\mathcal{Z}^h|^2\leq \frac{1}{h}\}} \mathcal{Q}_3\Big(G(x',
  hx_3)^{-1/2} \big(\frac{1}{h^{n+1}}J^h + o(1)\big)
 G(x', hx_3)^{-1/2}\Big)\dd x.
\end{split}
\end{equation*}
Since $B_0^{\trsp} (R^h)^{\trsp}\mathcal{Z}^h$ converges weakly
in $L^2(\Omega^1, \R^{3\times 3})$, up to a subsequence, to
$B_0^{\trsp}\bar R^{\trsp}\mathcal{Z}$, for some 
$\bar R\in SO(3)$ (which is an accumulation point of $\bar R^h$ in the
proof of Lemma \ref{compact}), the above results in:
\begin{equation}\label{comp2}
\begin{split}
& \liminf_{h\to 0}  \frac{1}{h^{2(n+1)}}\mathcal{E}^h(u^h)  \\ & \geq 
\frac{1}{2} \int_{\Omega^1} \mathcal{Q}_3\bigg(
\frac{x_3^{n+1}}{2(n+1)!} G(x', 0)^{-1/2}\Big(\sum_{k=1}^n\binom{n+1}{k}
(B_k)^{\trsp}B_{n+1-k} - \partial_{3}^{(n+1)}G(x',0)\Big) G(x',
0)^{-1/2} \\ & \qquad\qquad \qquad + G(x', 0)^{-1/2} \big(B_0^{\trsp} \bar
R^{\trsp}\mathcal{Z} \big)_{\sym}  G(x', 0)^{-1/2}\bigg)\dd x
\end{split}
\end{equation}

\medskip

{\bf 2.} We need to identify the relevant $2\times 2$ minor of the limiting term $\big(B_0^{\trsp} \bar
R^{\trsp}\mathcal{Z} \big)_{\sym}$ in (\ref{comp2}). We apply the
finite difference technique \cite{FJMhier} and consider the following
fields $\{f^{s,h}\in W^{1,2}(\Omega^1, \R^3)\}_{s>0, h\to 0}$:
\begin{equation*}
\begin{split}
f^{s,h}(x) & = \fint_0^sh (\bar R^h)^{\trsp}\mathcal{Z}^h(x', x_3+t) +
S^h\sum_{k=0}^n\frac{h^k(x_3+t)^k}{k!}B_k(x')\dd t ~ e_3 \\
& = \frac{1}{h^{n+1}s}(\bar R^h)^{\trsp}\big(u^h(x', h(x_3+s)) - u^h(x', hx_3)\big) -
\frac{1}{h^n}\fint_{0}^s \sum_{k=0}^n\frac{h^k(x_3+t)^k}{k!}\vec b_{k+1}\dd t.
\end{split}
\end{equation*}
where $S^h$ is defined in (\ref{Sh}). Recall that, as proved in Lemma
\ref{compact}, $\nabla V = (SB_0)_{3\times 2}$ and that $S$ is a.e. in
$so(3)$. It follows that the vector $\vec p$ defined in Theorem
\ref{teo2} (ii) must coincide with $SB_0e_3 = S\vec b_1$.
Consequently, using the first definition above it now easily follows that:
\begin{equation}\label{p}
f^{s,h}\to S\vec b_1 = \vec p \quad \mbox{ strongly in }\;
L^2(\Omega^1, \R^3).
\end{equation}
Using the second definition, we further compute the in-plane derivatives of $f^{s,h}$ for $j=1\ldots 2$:
\begin{equation*}
\begin{split}
\partial_j f^{s,h}(x) & = \frac{1}{sh^{n+1}}(\bar R^h)^{\trsp}
\big(\partial_j u^h(x', h(x_3+s)) - \partial_j u^h(x', hx_3)\big) - \frac{1}{h^n}\fint_0^s\sum_{k=0}^n
\frac{h^k(x_3+t)^k}{k!}\partial_k\vec b_{k+1}\dd t \\ & =
\frac{1}{s}(\bar R^h)^{\trsp} \Big(\mathcal{Z}^h(x', x_3+s) -
\mathcal{Z}^h(x', x_3)\Big)e_j  \\ & \qquad + \frac{1}{sh^{n+1}}(Id_3 + h^n
  S^h)\sum_{k=1}^n\frac{h^k}{k!}\big((x_3+s)^k - x_3^k\big)B_k e_j -
  \frac{1}{h^n}\fint_0^s\sum_{k=0}^n
  \frac{h^k(x_3+t)^k}{k!}\partial_j\vec b_{k+1}\dd t.
\end{split}
\end{equation*}
The first term in the right hand side above converges to $\displaystyle \frac{1}{s}\bar R^{\trsp}
\big(\mathcal{Z}(x', x_3+s) - \mathcal{Z}(x', x_3)\big)e_j$, weakly in
$L^2(\Omega^1, \R^{3})$, whereas the last two terms may be rewritten as:
\begin{equation*}
\begin{split}
& \frac{1}{sh^{n+1}}(Id_3 + h^n
S^h)\sum_{k=1}^n\frac{h^k}{k!}\big((x_3+s)^k - x_3^k\big) \partial_j\vec b_k -
  \frac{1}{sh^{n+1}}\sum_{k=1}^{n+1}\frac{h^k}{k!}\big((x_3+s)^k-x_3^k\big)\partial_j\vec b_{k}
\\ & = \frac{1}{sh}S^h \sum_{k=1}^n\frac{h^k}{k!} \big((x_3+s)^k-x_3^k\big)\partial_j\vec b_{k}
- \frac{1}{s(n+1)!} \big((x_3+s)^{n+1}-x_3^{n+1}\big)\partial_j\vec b_{n+1}\\ &
\qquad \rightharpoonup S\partial_j\vec b_1 - \frac{1}{s(n+1)!}
\big((x_3+s)^{n+1}-x_3^{n+1}\big)\partial_j\vec b_{n+1} 
\quad \mbox{ weakly in } \; W^{1,2}(\Omega^1, \R^3).
\end{split}
\end{equation*}
In conclusion, and recalling (\ref{p}), we obtain the following
convergence, weakly in $W^{1,2}(\omega, \R^3)$:
\begin{equation*}
\begin{split}
\partial_j f^{s,h}(x) \rightharpoonup \frac{1}{s}\bar R^{\trsp}
\big(\mathcal{Z}(x', x_3+s) - \mathcal{Z}(x', x_3)\big)e_j + 
S\partial_j\vec b_1 - \frac{1}{s(n+1)!}
\big((x_3+s)^{n+1}-x_3^{n+1}\big)\partial_j\vec b_{n+1} = \partial_j\vec p. 
\end{split}
\end{equation*}
We thus see that:
\begin{equation*}
\begin{split}
\bar R^{\trsp} \Big(\mathcal{Z}(x', x_3) - \mathcal{Z}(x', 0)\Big)e_j =
x_3 \big(\partial_j\vec p - S \partial_j\vec b_1\big) +
\frac{1}{(n+1)!} x_3^{n+1}\partial_j\vec b_{n+1} \quad \mbox{ for }\; j=1\ldots 2,
\end{split}
\end{equation*}
which finally yields:
\begin{equation}\label{compu}
\begin{split}
\Big(B_0^{\trsp}\bar R^{\trsp}\mathcal{Z}(x', x_3)\Big)_{2\times 2} =
& \;  \Big(B_0^{\trsp}\bar R^{\trsp}\mathcal{Z}(x', 0)\Big)_{2\times 2}  
\\ & + x_3 \Big((\nabla y_0)^{\trsp}\nabla \vec p + (\nabla V)^{\trsp}\nabla \vec b_1\Big) +
\frac{1}{(n+1)!} x_3^{n+1}(\nabla y_0)^{\trsp}\nabla \vec b_{n+1}.
\end{split}
\end{equation}

\medskip

{\bf 3.} We now compute the symmetric part of the trace term
$\big(B_0^{\trsp}\bar R^{\trsp}\mathcal{Z}(x', 0)\big)_{2\times 2,
  \sym}$ and conclude the proof of the Lemma.
It follows from (\ref{nabVh}) and the definition of $\mathcal{Z}^h$ in (\ref{Zh}) that:
$$\nabla V^h = h\int_{-1/2}^{1/2}(\bar R^h)^{\trsp}\mathcal{Z}^h_{3\times 2}\dd x_3
+ S^h \big(\nabla y_0 + \mathcal{O}(h^2)\big)$$
In virtue of (\ref{Zh}), (\ref{compu}) and (\ref{Ssym}), we obtain
convergence, weakly in $\mathbb{E}$:
$$\frac{1}{h}\big((\nabla y_0)^{\trsp}\nabla V^h\big)_{\sym}
\rightharpoonup \Big((\nabla y_0)^{\trsp} \bar R^{\trsp}
\mathcal{Z}(x',0)_{3\times 2}\Big)_{\sym} +
\int_{-1/2}^{1/2}\frac{x_3^{n+1}}{(n+1)!}\dd x_3 \big((\nabla
y_0)^{\trsp}\nabla \vec b_{n+1}\big)_{\sym},$$
which allows to conclude, by Lemma \ref{compact} (ii):
\begin{equation}\label{trace}
\big(B_0^{\trsp}\bar R^{\trsp}\mathcal{Z}(x', 0)\big)_{2\times 2, \sym}
= \mathbb{S} - \delta_{n+1} \big((\nabla
y_0)^{\trsp}\nabla \vec b_{n+1}\big)_{\sym}.
\end{equation}
This ends the proof of Lemma, in virtue of (\ref{comp2}), 
(\ref{compu}), (\ref{trace}) and recalling definitions (\ref{Qform}).
\end{proof}

\section{Relations between (i) and (iii) of Theorem
  \ref{teo1} and a proof of Theorem \ref{teo2} (iii)}\label{sec_geometry}

In this section we show the relation between the defining quantities
appearing in conditions (i) and (iii) of Theorem \ref{teo1}.
Equivalence of (i) and (iii) at $n=2$ has been shown in \cite{LL}, building on the previous results in
\cite{BLS, LRR}, while the proof of the general case will be carried
out by induction on $n\geq 2$. We start by introducing some notation that allows for
a systematic approach.

\medskip

Define the smooth matrix fields $\{\Gamma_a:\bar\Omega^1\to\R^{3\times
3}\}_{a=1\ldots 3}$ by setting their coefficients
$(\Gamma_a)_{bc}=\Gamma^b_{ac}$ to be the usual Christoffel symbols
$\displaystyle \Gamma^b_{ac} = \frac{1}{2}\sum_{m=1}^3G^{bm}\big(\partial_b G_{mc} +
\partial_cG_{mb} - \partial_mG_{bc}\big)$ of the metric $G$. Recall
the standard notation for the coefficients of the inverse:
$(G^{-1})_{ab}=G^{ab}$. Since the Levi-Civita connection is torsion-free,
it follows that $\Gamma_ae_b = \Gamma_b e_a$ for all $a,b=1\ldots 3$
and also, the Riemann curvature tensor is expressed by, for all $c,d=1\ldots 3$:
\begin{equation*}
\begin{split}
& \big[R^a_{b,cd}\big]_{a,b=1\ldots 3} = \big(\partial_c\Gamma_d + \Gamma_c\Gamma_d\big)
- \big(\partial_d\Gamma_c + \Gamma_d\Gamma_c\big),\\
& \big[R_{ab,cd}\big]_{a,b=1\ldots 3} = G \big[R^a_{b,cd}\big]_{a,b=1\ldots 3} ~.
\end{split}
\end{equation*}
Given a matrix field $F:\Omega^1\to\R^{3\times 3}$, we define:
$\nabla_a F = \partial_a F + \Gamma_a F$  for each $a=1\ldots 
3$, so that $(\nabla_a F)e_b$ coincides with the usual covariant
derivative of vector fields: $\nabla_a(Fe_b)$. It also follows that:
\begin{equation*}
\begin{split}
 \nabla_c\nabla_dF - \nabla_d\nabla_cF =
 \big[R^a_{b,cd}\big]_{a,b=1\ldots 3} F \qquad \mbox{and} \qquad \nabla_{a}(F_1F_2) = (\nabla_a F_1)F_2 + F_1\partial_a F_2.
\end{split}
\end{equation*}

We now partially reprove
the mentioned statements at $n=1,2$ for completeness of presentation.

\begin{lemma}\label{n1}
Assume that there exist smooth fields $y_0, \vec
b_1:\bar\omega\to\R^3$ such that the matrix field: $B_0=\big[\partial_1
y_0, \partial_2 y_0, \vec b_1\big]$ has positive determinant and such that:
\begin{equation*}
B_0^{\trsp}B_0=G(x',0) \;\mbox{ and } \;\big((\nabla y_0)^{\trsp}\nabla\vec
b_1\big)_{\sym}=\frac{1}{2}\partial_3G(x',0)_{2\times 2} \qquad \mbox{ for all }\; x'\in\omega. 
\end{equation*}
Then:
\begin{itemize}
\item[(i)] $\partial_iB_0 = B_0\Gamma_i$ for all $i=1\ldots 2$, and in
  particular: $\partial_i\vec b_1 = B_0\Gamma_3e_i$.
\item[(ii)] $R^a_{b,ij}(x',0) = R_{ab, ij}(x',0)=0$ for all
  $x'\in\omega$ and all $a,b=1\ldots 3$, $i,j=1\ldots 2$.
\item[(iii)] There exists a unique smooth field $\vec
  b_2:\bar\omega\to\R^3$ such that defining the matrix field $B_1=\big[\partial_1\vec
b_1, ~\partial_2 \vec b_1, ~\vec b_2\big]$, there holds: 
  $\displaystyle \big(B_0^{\trsp}B_1\big)_{\sym}=\frac{1}{2}\partial_3
  G(x',0)$ for all $x'\in \omega$. Moreover:
$$B_1 = B_0\Gamma_3 \quad \mbox{ and }\quad \partial_i\vec b_2 =
B_0\nabla_i\Gamma_3e_3 \quad \mbox{ for all }\; i=1\ldots 2.$$
\end{itemize}
\end{lemma}
\begin{proof}
{\bf 1.} One easily calculates, by a repeated use of the assumed identities,
that: $\langle \partial_i y_0, \partial_j\vec b_1\rangle = \partial_j
G_{i3} - \langle \partial_{ij} y_0, \vec b_1\rangle = \partial_j
G_{i3} - \partial_iG_{j3} + \langle \partial_{j} y_0, \partial_i\vec
b_1\rangle$ and thus: $\partial_3 G_{ij}= \langle \partial_{i} y_0, \partial_j\vec b_1\rangle  
+ \langle \partial_{j} y_0, \partial_i\vec b_1\rangle = \partial_j
G_{i3} - \partial_i G_{j3} + 2 \langle \partial_{j}
y_0, \partial_i\vec b_1\rangle$, for all $i,j=1\ldots 2$, where all
the identities are taken on $\omega\times \{0\}$. Thus:
\begin{equation}\label{alpha}
\langle \partial_{j} y_0, \partial_i\vec b_1\rangle =
\frac{1}{2}\big(\partial_3 G_{ij}+ \partial_i G_{j3} - \partial_j
G_{i3} = \big(G\Gamma_3\big)_{ji}\quad\mbox{ for all }\; i,j=1\ldots 2.
\end{equation}
Secondly: $\langle \partial_j y_0, \partial_{ik}y_0\rangle
= \partial_{i}G_{jk}- \langle \partial_k y_0, \partial_{ij}y_0\rangle
= \partial_iG_{jk}-\partial_jG_{ik}+ \langle \partial_i
y_0, \partial_{jk}y_0\rangle
= \partial_iG_{jk}-\partial_jG_{ik}+ \partial_k G_{ij} - \langle \partial_j
y_0, \partial_{ik}y_0\rangle$, which results in:
 \begin{equation*}
\langle \partial_{j} y_0, \partial_{ik}y_0\rangle =
\frac{1}{2}\big(\partial_i G_{jk}+ \partial_k G_{ij} - \partial_j
G_{ik}\big) = \big(G\Gamma_i\big)_{jk}\quad\mbox{ for all }\; i,j,k=1\ldots 2.
\end{equation*}
Thirdly, from (\ref{alpha}) we obtain:
\begin{equation*}
\langle \vec b_1, \partial_{ik}y_0\rangle
= \partial_iG_{k3}-\langle \partial_i\vec b_1, \partial_ky_0\rangle =
\frac{1}{2}\big(\partial_i G_{k3}+ \partial_k G_{i3} - \partial_3
G_{ik}\big) = \big(G\Gamma_i\big)_{3k}\quad\mbox{ for all }\; i,k=1\ldots 2.
\end{equation*}
Finally: $\langle \vec b_1, \partial_i \vec b_1\rangle =
\frac{1}{2}\partial_i G_{33} = \big(G\Gamma_i\big)_{33}$, so that the
last two identities yield:
$$B_0^{\trsp}\partial_i B_0 = G\Gamma_i \quad \mbox{ for all }\; i =
1\ldots 2, \quad \mbox{ on }\; \omega \times \{0\}.$$
This proves (i) and further: $\partial_i\vec b_1 = B_0\Gamma_ie_3 =
B_0\Gamma_3e_i$, as claimed. 

\medskip

{\bf 2.} Using (i) we compute:
\begin{equation*}
\begin{split}
0 & =\partial_{ij}B_0 - \partial_{ji}B_0 = \partial_i\big(B_0\Gamma_j\big) - 
\partial_j\big(B_0\Gamma_i\big) = B_0\Gamma_i\Gamma_j +
B_0\partial_i\Gamma_j - \big(B_0\Gamma_j\Gamma_i +
B_0\partial_j\Gamma_i\big) \\ & = -B_0 \big[R^k_{s,ij}(\cdot, 0)\big]_{k,s=1\ldots
  3}  \quad\mbox{ for all }\; i,j=1\ldots 2.
\end{split}
\end{equation*}
which implies (ii). For (iii), uniqueness of $\vec b_2$ is
obvious, while $\vec b_2 = B_0\Gamma_3e_3$ follows from the
requested defining identity, in view of (\ref{alpha}). The covariant
derivative formula is a consequence of (i).
\end{proof}

\begin{lemma}\label{n2}
Assume that there exist smooth fields $y_0, \vec
b_1, \vec b_2:\bar\omega\to\R^3$ such that the matrix field: $B_0=\big[\partial_1
y_0, ~\partial_2 y_0, ~\vec b_1\big]$ has positive determinant and that
together with $B_1=\big[\partial_1\vec b_1, ~\partial_2 \vec b_1, ~\vec b_2\big]$
it satisfies:
\begin{equation*}
\begin{split}
& B_0^{\trsp}B_0=G(x',0) \;\mbox{ and }
\;\big(B_0^{\trsp}B_1\big)_{\sym}=\frac{1}{2}\partial_3 G(x',0)\\
& \big((\nabla y_0)^{\trsp}\nabla\vec b_2\big)_{\sym} + (\nabla \vec
b_1)^{\trsp}\nabla \vec b_1 = \frac{1}{2}\partial_{33}G(x',0)_{2\times 2}
\qquad \mbox{  for all }\; x'\in\omega. 
\end{split}
\end{equation*}
Then:
\begin{itemize}
\item[(i)] $R_{ab, cd}(x',0)=0$ for all
  $x'\in\omega$ and all $a,b,c,d=1\ldots 3$.
\item[(ii)] There exists a unique smooth field $\vec
  b_3:\bar\omega\to\R^3$ such that defining the matrix field $B_2=\big[\partial_1\vec
b_2, ~\partial_2 \vec b_2, ~\vec b_3\big]$, there holds: 
  $\displaystyle \big(B_0^{\trsp}B_2\big)_{\sym} + B_1^{\trsp}B_1 =\frac{1}{2}\partial_{33}
  G(x',0)$ for all $x'\in \omega$. Moreover:
$$B_2 = B_0\nabla_3\Gamma_3 \quad \mbox{ and }\quad \partial_i\vec b_3 =
B_0\nabla_i\nabla_3\Gamma_3e_3 \quad \mbox{ for all }\; i=1\ldots 2.$$
\end{itemize}
\end{lemma}
\begin{proof}
Observe first that for all $a,b=1\ldots 3$ we have:
\begin{equation}\label{G2}
\begin{split}
\langle \partial_{33}G e_a, e_b\rangle & =\partial_3\big(\langle
G\Gamma_3 e_a, e_b\rangle + \langle G\Gamma_3 e_b, e_a\rangle\big) \\
& = \langle \nabla_3\Gamma_ae_3, Ge_b\rangle + \langle
\nabla_3\Gamma_be_3, Ge_a\rangle + 2 \langle G\Gamma_3e_a, \Gamma_3e_b\rangle.
\end{split}
\end{equation}
Consequently, and using Lemma \ref{n1} (iii), the last assumed condition is equivalent to:
\begin{equation*}
\begin{split}
0 & = \langle B_0e_i, \partial_j\vec b_2\rangle +
2\langle \partial_i\vec b_1, \partial_j\vec b_1\rangle + \langle
B_0e_j, \partial_i\vec b_2\rangle - \langle\partial_{33}G e_i, e_j\rangle
\\ & = \langle Ge_i, \nabla_j\Gamma_3 e_3\rangle + 2 \langle G\Gamma_3 e_i,
\Gamma_3 e_j\rangle + \langle Ge_j, \nabla_i\Gamma_3 e_3\rangle
\\ & \qquad\qquad 
- \Big(\langle \nabla_3\Gamma_je_3, Ge_i\rangle + \langle
\nabla_3\Gamma_ie_3, Ge_j\rangle + 2 \langle G\Gamma_3e_j, \Gamma_3e_j\rangle\Big)
\\ & = \langle Ge_i, \big[R^a_{3j3}(\cdot, 0)\big]_{a=1\ldots
  3}\rangle + \langle Ge_j, \big[R^a_{3j3}(\cdot, 0)\big]_{a=1\ldots 3}\rangle \\ &
= 2 R_{i3,j3}(\cdot, 0) \quad \mbox{on }\;\omega, \quad \mbox{for all }\; i,j=1\ldots 2.
\end{split}
\end{equation*}
The above proves (i), in virtue of Lemma \ref{n1} (ii) that guarantees
$R_{ab,ij}(\cdot, 0)=0$ for all $a,b=1\ldots 3$, $i,j=1\ldots 2$.
To show (ii), we observe that by Lemma \ref{n1} and by (i):
$$B_2e_i = \partial_i \vec b_2 = B_0\nabla_i \Gamma_3e_3 =
B_0\nabla_3\Gamma_ie_3 = B_0\nabla_3\Gamma_3e_i
\quad \mbox{for all }\; i,j=1\ldots 2$$
and also, $\displaystyle (B_0^{\trsp}B_2)_{\sym}
+ B_1^{\trsp}B_1 - \frac{1}{2}\partial_{33} G(x',0) =
(B_0^{\trsp}B_2)_{\sym} + \Gamma_3^{\trsp} G\Gamma_3 -
\Big((G\nabla_3\Gamma_3)_{\sym} + \Gamma_3^{\trsp} G\Gamma_3\Big)$, in
view of (\ref{G2}), so $B_2=B_0\nabla_3 \Gamma_3$ satisfies the defining
relation. Finally, $\partial_i\vec b_3
= \partial_i\big(B_0\nabla_3\Gamma_3e_3\big) =
B_0\nabla_i\nabla_3\Gamma_3e_3$ results from Lemma \ref{n1} (i).
\end{proof}

\medskip

We state the following two useful observations:

\begin{lemma}\label{lemmaG}
For all $n\geq 0$ there holds:
$$\partial_3^{(n+1)}G(x',0) = 2\Big(G\nabla_3^{(n)}\Gamma_3\Big)_{\sym}
+ \sum_{k=1}^n
\binom{n+1}{k}\big(\nabla_3^{(k-1)}\Gamma_3\big)^{\trsp}G\nabla_3^{(n-k)}\Gamma_3
\quad \mbox{for all }\; x'\in\omega. $$
\end{lemma}
\begin{proof}
The proof follows by induction. For $n=0$, the statement is obviously
true. Assume that it is true for some $n-1$, then:
\begin{equation*}
\begin{split}
\partial_3^{(n+1)}G(x',0) & = \partial_3\bigg(2\big(G\nabla_3^{(n-1)}\Gamma_3\big)_{\sym}
+ \sum_{k=1}^{n-1}
\binom{n}{k}\big(\nabla_3^{(k-1)}\Gamma_3\big)^{\trsp}G\nabla_3^{(n-1-k)}\Gamma_3\bigg) \\ &
= G\nabla_3^{(n)}\Gamma_3 + \big(\nabla_3^{(n)}\Gamma_3\big)^{\trsp}G
+ \Gamma_3^{\trsp}G\nabla_3^{(n-1)}\Gamma_3 +
\big(\nabla_3^{(n-1)}\Gamma_3\big)^{\trsp}G\Gamma_3 \\ & \quad +
\sum_{k=1}^{n-1}\binom{n}{k} \Big(
\big(\nabla_3^{(k-1)}\Gamma_3\big)^{\trsp}G\nabla_3^{(n-k)}\Gamma_3 
+ \big(\nabla_3^{(k)}\Gamma_3\big)^{\trsp}G\nabla_3^{(n-k-1)}\Gamma_3\Big)
\\ & = G\nabla_3^{(n)}\Gamma_3 + \big(\nabla_3^{(n)}\Gamma_3\big)^{\trsp}G
+ \Gamma_3^{\trsp}G\nabla_3^{(n-1)}\Gamma_3 +
\big(\nabla_3^{(n-1)}\Gamma_3\big)^{\trsp}G\Gamma_3 \\ & \quad +
\sum_{k=1}^{n-1}\Big( \binom{n}{k} +\binom{n}{k-1}\Big) 
\big(\nabla_3^{(k-1)}\Gamma_3\big)^{\trsp}G\nabla_3^{(n-k)}\Gamma_3 \\ &
\quad - \binom{n}{0}\Gamma_3^{\trsp}G \nabla_3^{(n-1)}\Gamma_3 + 
\binom{n}{n-1}\big(\nabla_3^{(n-1)}\Gamma_3\big)^{\trsp}G \Gamma_3.
\end{split}
\end{equation*}
Collecting all the terms and recalling that $\binom{n}{k} +\binom{n}{k-1}=\binom{n+1}{k}$
implies the result.
\end{proof}

\begin{lemma}\label{Riem_zero}
Assume that $R_{12,12}(x',0)=R_{12,13}(x',0)=R_{12,23}(x',0)=0 $ for all
$x'\in \omega$ and also that $\partial_3^{(k)}R_{i3,j3}(x',0)=0$ for
all $k=0\ldots n$, all $i,j=1\ldots 2$ and all $x'\in\omega$. Then all
the mixed partial derivatives of both $R_{ab,cd}$ and $R^a_{b,cd}$, of any order up to $n$,
are zero on $\omega$, for all $a,b,c,d=1\ldots 3$.
\end{lemma}
\begin{proof}
The proof proceeds by induction on $n$. For $n=0$ the result is
obviously true. Assume that it is true for some $n\geq 0$ and let the result
assumption at $n+1$ hold. Then:
\begin{equation*}
\begin{split}
& \partial^{(k)}R^a_{b,cd}(x',0) = \partial^{(k)}R_{ab,cd}(x',0) = 0
\quad \mbox{ for all }\; k=0\ldots n, \quad a,b,c,d=1\ldots 3,\\ 
& \partial_3^{(n+1)}R_{i3,j3}(x',0) = 0 \qquad \qquad\qquad \quad \;
\; \mbox{ for all }\; i,j=1\ldots 2, \quad x'\in \omega,
\end{split}
\end{equation*}
and we need to show that any partial derivatives of order $n+1$, of the Riemann
tensor's components is zero on $\omega$. This is certainly true for partial derivatives
containing $\partial_i$ for some $i=1\ldots 2$, so it suffices to
prove the claim for $\partial_3^{(n+1)}$. Below, we consider various combinations
of indices $i,j=1\ldots 2$ and $a,b=1\ldots 3$. Firstly:
\begin{equation}\label{a}
\partial_3^{(n+1)}R^a_{b,ij}= \partial_3^{(n)}\nabla_3 R^a_{b,ij}
= \partial_3^{(n)}\big(-\nabla_i R^a_{b,j3}-\nabla_j R^a_{b,31}\big)
= \partial_3^{(n)}\big(-\partial_i R^a_{b,j3}-\partial_j R^a_{b,31}\big)=0,
\end{equation}
where we used the induction assumption in the first and the third
equalities and  the second Bianchi identity in the second one.
Secondly:
\begin{equation}\label{b}
\partial_3^{(n+1)}R_{ab,ij}= \partial_3^{(n+1)}\langle
\big[G_{ap}\big]_{p=1\ldots 3}, \big[ R^p_{b,ij}\big]_{p=1\ldots
3}\rangle = \langle \big[G_{ap}\big]_{p=1\ldots 3}, \big[\partial_3^{(n+1)}R^p_{b,ij}\big]_{p=1\ldots
3}\rangle =0,
\end{equation}
where we used the induction assumption and (\ref{a}) in the last
equality. Thirdly:
\begin{equation}\label{c}
\partial_3^{(n+1)}R^a_{b,i3}= \partial_3^{(n+1)}\langle
\big[G^{ap}\big]_{p=1\ldots 3}, \big[ R_{pb,i3}\big]_{p=1\ldots
3}\rangle = \langle \big[G^{ap}\big]_{p=1\ldots 3}, \big[\partial_3^{(n+1)}R_{pb,i3}\big]_{p=1\ldots
3}\rangle =0,
\end{equation}
by using (\ref{b}) and the result assumption at $n+1$, in the last equality.
Finally: $\partial_3^{(n+1)}R_{ab,cd} = 0$ by (\ref{b}) and the result assumption.
\end{proof}

\medskip

The following is the main result of this section:

\begin{lemma}\label{industep}
Fix $n\geq 2$. Assume that there exist smooth $y_0, \{\vec
b_k\}_{k=1}^n:\bar\omega\to\R^3$ such that the matrix fields: $B_0=\big[\partial_1
y_0, ~\partial_2 y_0, ~\vec b_1\big]$ with positive determinant and 
$\{B_k=\big[\partial_1\vec b_k, ~\partial_2 \vec b_k, ~\vec b_{k+1}\big]\}_{k=1}^{n-1}$, satisfy:
\begin{equation*}
\begin{split}
& \sum_{k=0}^m\binom{m}{k}B_k^{\trsp}B_{m-k} - \partial_3^{(m)} G(x',0) =
0\quad \mbox{  for all }\; m=0\ldots n-1,\\
& 2\big((\nabla y_0)^{\trsp}\nabla\vec b_n\big)_{\sym} + \sum_{k=1}^{n-1}\binom{n}{k}(\nabla \vec
b_k)^{\trsp}\nabla \vec b_{n-k} = \partial_{3}^{(n)}G(x',0)_{2\times 2}
\qquad \mbox{  for all }\; x'\in\omega. 
\end{split}
\end{equation*}
Then:
\begin{itemize}
\item[(i)] Condition in Theorem \ref{teo1} (i) holds.
\item[(ii)] There exists a unique smooth field $\vec
  b_{n+1}:\bar\omega\to\R^3$ such that defining the matrix field $B_n=\big[\partial_1\vec
b_{n}, ~\partial_{2} \vec b_{n}, ~\vec b_{n+1} \big]$, there holds: 
  $\displaystyle \sum_{k=0}^n\binom{n}{k}B_k^{\trsp}B_{n-k} = \partial_{3}^{(n)}
  G(x',0)$ for all $x'\in \omega$. Moreover:
$$B_n = B_0\nabla_3^{(n-1)}\Gamma_3 \quad \mbox{ and }\quad \partial_i\vec b_{n+1} =
B_0\nabla_i\nabla_3^{(n-1)}\Gamma_3e_3 \quad \mbox{ for all }\; i=1\ldots 2.$$
\end{itemize}
\end{lemma}
\begin{proof}
{\bf 1.} The proof proceeds by induction. The statement at $n=2$ has been shown
in Lemma \ref{n2}. We now assume it to be true for some $n\geq 2$. By
Lemma \ref{Riem_zero}, we get:
\begin{equation}\label{rr}
\mbox{\begin{minipage}{13cm}
All mixed partial derivatives up to order $n-2$, of all components of
the Riemann curvature tensor, are $0$ at $\omega\times \{0\}$.
\end{minipage}}
\end{equation}
Since $B_k=\big[\partial_1\vec b_n, ~\partial_2\vec b_n, ~\vec
b_{n+1}\big]$ with $\vec b_{n+1}$ as in (ii), and recalling Lemma
\ref{lemmaG}, we obtain for all $x'\in\omega$:
\begin{equation}\label{main_curv}
\begin{split}
2\big((&\nabla y_0)^{\trsp}\nabla\vec b_{n+1}\big)_{\sym} + \sum_{k=1}^{n}\binom{n+1}{k}(\nabla \vec
b_k)^{\trsp}\nabla \vec b_{n+1-k} - \partial_{3}^{(n+1)}G(x',0)_{2\times 2}
\\ & = 2\big((\nabla y_0)^{\trsp}\nabla\vec b_{n+1}\big)_{\sym} +
\sum_{k=1}^{n}\Big(B_k^{\trsp}B_{n+1-k}\Big)_{2\times 2} \\ &
\qquad\qquad\qquad\qquad \qquad - \bigg(
2\big(G\nabla_3^{(n)}\Gamma_3\big)_{\sym}
+ \sum_{k=1}^n
\binom{n+1}{k}\big(\nabla_3^{(k-1)}\Gamma_3\big)^{\trsp}G\nabla_3^{(n-k)}\Gamma_3\bigg)_{2\times 2}
\\ & = 2\Big((\nabla y_0)^{\trsp}\nabla\vec b_{n+1} - G\nabla_3^{(n)}\Gamma_3\Big)_{\sym} 
= 2\Big[\langle Ge_i, \nabla_j\nabla_3^{(n-1)}\Gamma_3e_3 -
\nabla_3^{(n)}\Gamma_3e_j\rangle\Big]_{i,j=1\ldots 2, \;\sym} \\ & 
= 2\Big[\langle Ge_i, \nabla_j\nabla_3^{(n-1)}\Gamma_3e_3 -
\nabla_3^{(n)}\Gamma_je_3\rangle\Big]_{i,j=1\ldots 2, \;\sym}.
\end{split}
\end{equation}
By (\ref{rr}) we can consecutively swap the order of all the covariant
derivatives on $\omega\times \{0\}$ in:
$$ \nabla_j\nabla_3^{(n-1)}\Gamma_3 =
\nabla_3\nabla_j\nabla_3^{(n-2)}\Gamma_3 =
\nabla_3^{(2)}\nabla_j\nabla_3^{(n-3)}\Gamma_3 =
\nabla_3^{(3)}\nabla_j\nabla_3^{(n-4)}\Gamma_3 = (\ldots) =
\nabla_3^{(n-1)}\nabla_j\Gamma_3, $$
so that:
\begin{equation}\label{swap} 
\nabla_j\nabla_3^{(n-1)}\Gamma_3 - \nabla_3^{(n)}\Gamma_j = \nabla_3^{(n-1)}\big(\nabla_j\Gamma_3
- \nabla_3\Gamma_j\big) =
\nabla_3^{(n-1)}\big[R^a_{b,j3}(x',0)\big]_{a,b=1\ldots 3}.
\end{equation}
In conclusion, using (\ref{rr}) again, the formula in (\ref{main_curv}) becomes:
\begin{equation}\label{main_curv2}
\begin{split}
2\big((\nabla y_0)^{\trsp}\nabla\vec b_{n+1}&\big)_{\sym} + \sum_{k=1}^{n}\binom{n+1}{k}(\nabla \vec
b_k)^{\trsp}\nabla \vec b_{n+1-k} - \partial_{3}^{(n+1)}G(x',0)_{2\times 2}
\\ & = \Big[\langle Ge_i, \big[\partial_3^{(n-1)}R^a_{3,j3}(x',0)\big]_{a=1\ldots 3}\rangle + 
 \langle Ge_j, \big[\partial_3^{(n-1)}R^a_{3,i3}(x',0)\big]_{a=1\ldots 3}\rangle\Big]_{i,j=1\ldots 2} \\ & =
 2\Big[\partial_3^{(n-1)}R_{i3,j3}(x',0)\Big]_{i,j=1\ldots 2} \qquad \mbox{  for all }\; x'\in\omega,
\end{split}
\end{equation}
proving (i) in view of the second assumption at $n+1$.

\medskip

{\bf 2.} For (ii), observe that $B_{n+1}$ is indeed
uniquely defined, by choosing $\vec b_{n+2} = B_{n+1}e_3$ such that:
$$\sum_{k=0}^{n+1}\binom{n+1}{k}B_k^{\trsp}B_{n+1-k} = \partial_3^{(n+1)} G(x',0) 
 \quad \mbox{ for all }\; x'\in\omega,$$
since the principal $2\times 2$ minors of both sides in the above
formula coincide by assumption. Further, by (\ref{swap}) and
the already established (i) at $n+1$, we get:
\begin{equation*}
\begin{split}
B_{n+1}e_i & = \partial_i\vec b_{n+1}=\partial_i\Big(
B_0\nabla_3^{(n-1)}\Gamma_3e_3\Big) = B_0
\nabla_i\nabla_3^{(n-1)}\Gamma_3e_3 \\ & =
B_0 \nabla_3^{(n)}\Gamma_ie_3+ \nabla_3^{(n-1)}\big[R^a_{3,i3}(x',0)\big]_{a=1\ldots 3}
= B_0 \nabla_3^{(n)}\Gamma_ie_3 \\ & = B_0 \nabla_3^{(n)}\Gamma_3e_i \qquad
\mbox{ for all }\; i=1\ldots 2 \quad \mbox{and all }\; x'\in\omega.
\end{split}
\end{equation*}
Hence, there must be $\vec b_{n+1} = B_0\nabla_3^{(n)}\Gamma_3$, as
claimed. This ends the proof of the Lemma.
\end{proof}

\medskip

We note that the argument in the proof above leading to
(\ref{main_curv2}), automatically gives:

\begin{cor}\label{t1iii}
For any $n\geq 1$, condition (iii) in Theorem \ref{teo1} implies the formula (\ref{riem}).
\end{cor}

\section{The end of proof of Theorem \ref{teo2} and a proof of Theorem
\ref{teo1}}\label{sec_I2(n+1)}

The following statement concludes the proof of Theorem \ref{teo2}, assuming (iii) of Theorem \ref{teo1}:

\begin{lemma}\label{fin_strange} 
In the context of Lemma \ref{compact}, there holds: $\displaystyle \liminf_{h\to
  0}\frac{1}{h^{2(n+1)}}\mathcal{E}^h(u^n) \geq \mathcal{I}_{2(n+1)}(V)$.
\end{lemma}
\begin{proof}
By Lemma \ref{compact} and Corollary \ref{t1iii}, we get:
\begin{equation*}
\begin{split}
\liminf_{h\to 0}  \frac{1}{h^{2(n+1)}}\mathcal{E}^h(u^h) \geq  \frac{1}{2} \int_{\Omega^1}\mathcal{Q}_2\bigg(x',~
& \mathbb{S} - \delta_{n+1} \big((\nabla y_0)^{\trsp}\nabla \vec b_{n+1}\big)_{\sym} 
+ x_3 \big((\nabla y_0)^{\trsp}\nabla \vec p + (\nabla V)^{\trsp}\nabla \vec b_1\big)
\\ & + \frac{x_3^{n+1}}{(n+1)!} \big[\partial_3^{(n-1)}R_{i3,j3}(x',0)\big]_{i,j=1\ldots 2}\bigg)\dd x.
\end{split}
\end{equation*}
Denoting the $x'$-dependent tensor terms at different powers of $x_3$
in the integrand above by $I,II$ and $III$, and recalling the
definition of $\delta_{n+1}$ in (\ref{deldef}), the right hand side becomes:
\begin{equation*}
\begin{split}
\frac{1}{2} & \int_{\Omega^1}\mathcal{Q}_2\big(x',~ I+x_3 II +
x_3^{n+1}III \big)\dd x \\ & = 
\frac{1}{2} \int_{\omega}\mathcal{Q}_2\Big(x',~ I+
\big(\int_{-1/2}^{1/2} x_3^{n+1}\dd x_3\big) III \Big) + 
\frac{1}{24} \mathcal{Q}_2\Big(x',~ II+
12\big(\int_{-1/2}^{1/2} x_3^{n+2}\dd x_3\big) III \Big)\\ & 
\qquad \quad + \frac{1}{2} \bigg(\big(\int_{-1/2}^{1/2} x_3^{2n+2}\dd x_3\big) -
\big(\int_{-1/2}^{1/2} x_3^{n+1}\dd x_3\big)^2 - 12 \big(\int_{-1/2}^{1/2} x_3^{n+2}\dd x_3\big)^2\bigg)
 \mathcal{Q}_2\big(x',~ III \big)\dd x'
\\ & = \frac{1}{2} \int_{\omega}\mathcal{Q}_2\bigg(x',~
\mathbb{S} - \delta_{n+1} \big((\nabla y_0)^{\trsp}\nabla \vec b_{n+1}\big)_{\sym} 
+ \delta_{n+1} \big[\partial_3^{(n-1)}R_{i3,j3}(x',0)\big]_{i,j=1\ldots 2}\bigg)\dd x'
\\ & \quad + \frac{1}{24} \int_{\omega}\mathcal{Q}_2\bigg(x',~
(\nabla y_0)^{\trsp}\nabla \vec p + (\nabla V)^{\trsp}\nabla \vec b_1+
\alpha_n \big[\partial_3^{(n-1)}R_{i3,j3}(x',0)\big]_{i,j=1\ldots
  2}\bigg)\dd x' \\ & \quad + \gamma_n \int_{\omega}\mathcal{Q}_2\bigg(x',~
\big[\partial_3^{(n-1)}R_{i3,j3}(x',0)\big]_{i,j=1\ldots 2}\bigg)\dd x',
\end{split}
\end{equation*}
where by a direct calculation one easily checks that the numerical coefficients $\alpha_n$
and $\gamma_n$ have the form (\ref{albegam}). Further, since
$\mathbb{S} - \delta_{n+1} \big((\nabla y_0)^{\trsp}\nabla \vec
b_{n+1}\big)_{\sym} \in\mathcal{S}_{y_0}$, the first term in the right
hand side above is bounded from below by:
\begin{equation*}
\begin{split}
\frac{1}{2} \dist^2_{\mathcal{Q}_2}\Big(\delta_{n+1}
\big[\partial_3^{(n-1)}R_{i3,j3}(x',0)\big]_{i,j=1\ldots 2},
~\mathcal{S}_{y_0}\Big) & = \frac{\delta^2_{n+1}}{2} \dist^2_{\mathcal{Q}_2}\Big(
\big[\partial_3^{(n-1)}R_{i3,j3}(x',0)\big]_{i,j=1\ldots 2},
~\mathcal{S}_{y_0}\Big) \\ & = \frac{\delta^2_{n+1}}{2} \Big\|\mathbb{P}_{\mathcal{S}_{y_0}^{\perp}}\big(
\big[\partial_3^{(n-1)}R_{i3,j3}(x',0)\big]_{i,j=1\ldots 2}\big)\Big\|_{\mathcal{Q}_2}^2.
\end{split}
\end{equation*}
Decomposing the third term into:
$$\gamma_n \Big\|\mathbb{P}_{\mathcal{S}_{y_0}^{\perp}}\big(
\big[\partial_3^{(n-1)}R_{i3,j3}(x',0)\big]_{i,j=1\ldots 2}\big)\Big\|_{\mathcal{Q}_2}^2
+ \gamma_n\Big\|\mathbb{P}_{\mathcal{S}_{y_0}}\big(
\big[\partial_3^{(n-1)}R_{i3,j3}(x',0)\big]_{i,j=1\ldots 2}\big)\Big\|_{\mathcal{Q}_2}^2,$$
the claim follows by checking that: $\displaystyle
\frac{\delta^2_{n+1}}{2} + \gamma_n = \beta_n$ in (\ref{albegam}).
\end{proof}

\medskip

We are now ready to give:

\medskip

\noindent {\bf A proof of Theorem \ref{teo1}.}
The proof is carried out by induction on $n\geq 2$. When $n=2$, then
(i) is equivalent with (iii) by facts recalled in the preliminary
discussion in section \ref{sec_begin}. Condition (iii) implies (ii) by
Lemma \ref{3to2}, whereas (ii) implies (i) again in view of (\ref{h40}).

\smallskip

Assume now the equivalence of the three conditions at some $n\geq
2$. We want to show the equivalence at $n+1$. Condition (i) implies
(iii) by Corollary \ref{t1iii}. Condition (iii) implies (ii) by Lemma
\ref{3to2}. Finally, assuming (ii) at $n+1$ allows to write:
\begin{equation*}
\begin{split} 
0 & = \lim_{h\to 0} \frac{1}{h^{2(n+1)}}\inf \mathcal{E}^h =
\lim_{h\to 0} \frac{1}{h^{2(n+1)}}\mathcal{E}^h(u^h)\geq
\mathcal{I}_{2(n+1)}(V) \\ & \geq \gamma_n \cdot
\Big\|\big[\partial_3^{(n-1)}R_{i3,j3}(x',0)\big]_{i,j=1\ldots
  2}\Big\|_{\mathcal{Q}_2}^2,
\end{split}
\end{equation*}
for some infimizing sequence $\{u^h\in W^{1,2}(\Omega^h,\R^3)\}_{h\to 0}$ and a resulting $V$
from Theorem \ref{teo2}. This establishes (i) at $n+1$, in view of the inductive assumption.
\endproof

\medskip

For completeness, we state the following auxiliary observations:

\begin{lemma}\label{lem_extra}
In the context of Theorem \ref{teo2}, we have:
\begin{itemize}
\item[(i)] The bending term $(\nabla y_0)^{\trsp}\nabla \vec p + (\nabla V)^{\trsp}\nabla \vec b_1$
is symmetric and it equals: $$\displaystyle \Big[\big\langle \Gamma_je_i,
\left[\begin{array}{c}(\nabla V)^{\trsp}\vec b_1 \\  0\end{array}\right]\big\rangle -
\langle \partial_{ij} V, \vec b_1\rangle \Big]_{i,j=1\ldots 2}. $$
\item[(ii)] Under any of the equivalent conditions in Theorem
  \ref{teo1} at $n+1$, we have: 
$$\mathrm{Ker}\;\mathcal{I}_{2(n+1)} = 
\big\{Sy_0+c; ~ S\in so(3), ~ c\in\mathbb{R}^3\big\},$$ 
and the following coercivity estimate holds:
\begin{equation*}
\dist^2_{W^{2,2}(\omega, \mathbb{R}^3)} \big(V,~ \mathrm{Ker}\; \mathcal{I}_{2(n+1)}\big) \leq
C \mathcal{I}_{2(n+1)}(V)\qquad\mbox{ for all }\; V\in\mathcal{V}_{y_{y_0}}
\end{equation*}
with a constant $C>0$ that depends on $G,\omega$ and $W$ but is
independent of $V$.
\end{itemize}
\end{lemma}
\begin{proof}
The symmetry of the bending term in (i) follows from:
\begin{equation*}
\begin{split}
\langle\partial_iy_0, \partial_j\vec p\rangle +
\langle\partial_iV, \partial_j\vec b_1\rangle & = \partial_j\big( 
\langle\partial_iy_0, \vec p\rangle + \langle\partial_iV, \vec b_1\rangle\big)
- \big(\langle\partial_{ij}y_0, \vec p\rangle +
\langle\partial_{ij}V, \vec b_1\rangle\big) \\ & = - \langle\partial_{ij}y_0, \vec p\rangle -
\langle\partial_{ij}V, \vec b_1\rangle \qquad \mbox{ for all } i,j=1\ldots 2.
\end{split}
\end{equation*}
The coercivity statement in (ii) has been proved in \cite[Theorems 8.2, 8.3]{LL}.
\end{proof}

\section{A proof of Theorem \ref{teo3}}\label{sec_recseq}

In this section, we prove the upper bound result of Theorem \ref{teo3}.
In view of the already established Theorem \ref{teo1}, it suffices to show:

\begin{lemma}\label{construct_rec}
Fix $n\geq 2$ and assume  condition (iii) in Theorem \ref{teo1}. Let
$V\in \mathcal{V}_{y_0}$ be a first order isometry displacement as in
(\ref{Vy0}). Then, there exists a sequence $\{u^h\in
W^{1,2}(\Omega^h, \R^3)\}_{h\to 0}$ of deformations satisfying (\ref{goodappro}),
and such that: $\displaystyle \liminf_{h\to 0 }
\frac{1}{h^{2(n+1)}}\mathcal{E}^h(u^h) = \mathcal{I}_{2(n+1)}(V)$.
\end{lemma}
\begin{proof}
{\bf 1.} Denote $\displaystyle Y(x',x_3) = y_0 + \sum_{k=1}^{n+1}\frac{x_3^k}{k!}\vec b_k$ and define:
\begin{equation}\label{uh}
\begin{aligned}
u^h(x', x_3) = & \, \, Y(x', x_3)+ h^n v^h(x')+h^{n+1}w^h(x')+
h^nx_3\vec p^h(x') + h^{n+1}x_3 \vec q^h(x') \\
& + \frac{x_3^{n+2}}{(n+2)!} \vec k_0(x') + h^n\frac{x_3^2}{2}\vec r^h(x')
\qquad \mbox{for all }\, (x',x_3)\in \Omega^h.
\end{aligned}
\end{equation}
We now introduce terms in the above expansion.
For a fixed small $\varepsilon >0$, the truncated sequence $\{v^h\in
W^{2,\infty}(\omega, \R^3)\}_{h\to 0}$ is chosen according to the
standard construction in \cite{FJM02} (see also references therein), in a way that: 
\begin{equation}\label{vh}
\begin{split}
& v^h\to V\quad\mbox{ strongly in }\,W^{2,2}(\omega,\R^3) ~~ \mbox{ as }\,h\to 0,\\
& h^n\|v^h\|_{W^{2,\infty}(\omega, \R^3)}\leq \varepsilon\quad\mbox{ and
}\quad\lim_{h\to 0}\frac{1}{h^{2n}}\,\big|\{x'\in \omega;~ v^h(x')\neq
V(x')\}\big|=0. 
\end{split}
\end{equation}
The sequence $\{\vec p^h\in W^{1,\infty}(\omega, \R^3)\}_{h\to 0}$ is defined by:
\begin{equation}\label{ph} 
B_0^{\trsp}\vec p^h =\left[\begin{matrix}
-(\nabla v^h)^{\trsp}\vec b_1\\ 0\end{matrix}\right]\qquad \mbox{so
that} \qquad \Big(B_0^{\trsp}\big[\nabla v^h, ~ \vec
p^h\big]\Big)_{\sym} = \Big((\nabla y_0)^{\trsp}\nabla v^h\Big)_{\sym}^*.
\end{equation}
The sequence $\{w^h\in \mathcal{C}^{\infty}(\bar \omega, \R^3)\}_{h\to
  0}$ is such that, recalling (\ref{deldef}): 
\begin{equation}\label{wh}
\begin{split}
& \Big((\nabla y_0)^{\trsp}\nabla w^h\Big)_{\sym}\to
- \delta_{n+1} \mathbb{P}_{\mathcal{S}_{y_0}}\Big(
\big[\partial_3^{(n-1)}R_{i3j3}\big]_{i,j=1\ldots 2}\Big)\\
& \quad\qquad\qquad\qquad\qquad \mbox{
  strongly in }\, \mathbb{E}=\LL^2(\omega, \R^{2\times 2}_{\sym}) ~~\mbox{ as } h\to 0, \\
& \lim_{h\to 0} h^{1/2}\|w^h\|_{\WW^{2,\infty}(\omega, \R^3)}=0.
\end{split}
\end{equation}

Finally, $\vec k_0\in \mathcal{C}^{\infty}(\bar\omega, \R^3)$
and $\{\vec q^h \in \mathcal{C}^{\infty}(\bar \omega, \R^3)\}_{h\to 0}$, $\{\tilde r^h\in
\LL^{\infty}(\omega, \R^3)\}_{h\to 0}$ are defined by:
\begin{equation}\label{qh}
\begin{aligned}
2B_0^{\trsp}  \vec k_0=&\, c\Big(x',  2\big((\nabla
y_0)^{\trsp}\nabla \vec b_{n+1}\big)_{\sym} +
\sum_{k=1}^n\binom{n+1}{k}(\nabla \vec b_k)^{\trsp}\nabla\vec
b_{n+1-k} - \partial_3^{(n+1)}G(x',0)_{2\times 2}\Big) \vspace{1mm} \\  & - \left[\begin{matrix}
\displaystyle 2 \sum_{k=0}^n\binom{n+1}{k}(\nabla \vec b_{n+1-k})^{\trsp}\nabla\vec
b_{k+1} \vspace{1mm}  \\ \displaystyle \sum_{k=1}^n\binom{n+1}{k}(\nabla \vec b_{k+1})^{\trsp}\nabla\vec
b_{n+2-k}  \end{matrix}\right] + \left[\begin{matrix} \displaystyle
    2\partial_3^{(n+1)}G(x',0)_{31,32} \vspace{1mm} \\
\displaystyle \partial_3^{(n+1)}G(x',0)_{33}\end{matrix}\right],  \\
B_0^{\trsp}\vec q^h = &\, c\Big(x',\big((\nabla
y_0)^{\trsp}\nabla w^h\big)_{\sym}\Big) - \left[\begin{matrix}
(\nabla w^h)^{\trsp}\vec b_1 \vspace{1mm} \\ 0\end{matrix}\right],\\
B_0^{\trsp}\tilde r^h =&\, c\Big(x', (\nabla
y_0)^{\trsp}\nabla \vec p^h+ (\nabla v^h)^{\trsp}\nabla
\vec b_1\Big)-\left[\begin{matrix} 
(\nabla v^h)^{\trsp}\vec b_2 \vspace{1mm} \\ \langle \vec p^h, \vec b_2\rangle\end{matrix}\right].
\end{aligned}
\end{equation}
Further, we choose $\{\vec r^h\in \mathcal{C}^\infty(\bar\omega,
\R^3)\}_{h\to  0}$ to satisfy, in view of (\ref{vh}):
\begin{equation}\label{rhh}
\lim_{h\to 0}\|\vec r^h-\tilde r^h\|_{\LL^2(\omega,
  \R^3)}=0\quad\mbox{ and }\quad \lim_{h\to
  0} {h}^{1/2}\|\vec r^h\|_{\WW^{1,\infty}(\omega,\R^3)}=0. 
\end{equation}

\medskip

{\bf 2.} By (\ref{wh}) and (\ref{rhh}) we easily deduce
(\ref{goodappro}). Compute now, for all rescaled variables $(x',x_3)\in\Omega^1$:
\[
\begin{aligned}
\nabla u^h(x',hx_3)= & \; h^n \big[\nabla v^h,~ \vec p^h\big] +\sum_{k=0}^n \frac{h^kx_3^k}{k!}B_k +
\frac{h^{n+1}x_3^{n+1}}{(n+1)!}\big[ \partial_1\vec b_{n+1},
~\partial_2\vec b_{n+1},~ \vec k_0\big] \\ & 
+ h^{n+1}x_3 \big[ \nabla \vec
p^h,~ \vec r^h\big] + h^{n+1}\big[\nabla w^h,~ \vec q^h\big] +
\mathcal{O}(h^{n+2}) \big(1+ |\nabla \vec q^h| + |\nabla \vec r^h|\big).
\end{aligned}
\]
Consequently, it follows that for $h$ small enough we have:
\[
\dist\big((\nabla u^h) G^{-1/2}, SO(3)\big)\leq C\big( |\nabla u^h -
B_0| + h\big) \leq C\epsilon,
\]
which justifies writing,  by Taylor's expansion of $W$ and taking $\epsilon\ll 1$:
\begin{equation*}
\begin{split}
W\big((\nabla u^h) G^{-1/2}\big) & = W\Big(\sqrt{Id_3 +
  G^{-1/2}\big((\nabla u^h)^{\trsp}\nabla u^h - G\big)G^{-1/2}}\Big) \\ & =  W\Big(Id_3 +
 \frac{1}{2} G^{-1/2}\big((\nabla u^h)^{\trsp}\nabla u^h - G\big)G^{-1/2} +
 \mathcal{O}\big(|(\nabla u^h)^{\trsp}\nabla u^h - G|^2\big) \Big) \\ & = 
W\Big(Id_3 + \frac{1}{2} G(x',0)^{-1/2}\big((\nabla u^h)^{\trsp}\nabla u^h -
G\big)G(x',0)^{-1/2} \\ & \qquad\qquad\qquad \qquad\qquad +
 \mathcal{O}\big(h|(\nabla u^h)^{\trsp}\nabla u^h - G|\big) +
 \mathcal{O}\big(|(\nabla u^h)^{\trsp}\nabla u^h - G|^2\big) \Big) \\ & 
= \frac{1}{8} \mathcal{Q}_3\Big(G(x',0)^{-1/2}\big((\nabla u^h)^{\trsp}\nabla
u^h - G\big)G(x',0)^{-1/2}\Big) \\ & \qquad\qquad\qquad \qquad\qquad +
 \mathcal{O}\big(h|(\nabla u^h)^{\trsp}\nabla u^h - G|^2\big) +
 \mathcal{O}\big(|(\nabla u^h)^{\trsp}\nabla u^h - G|^3\big).
\end{split}
\end{equation*}
This implies that:
\begin{equation}\label{caca}
\begin{split}
\frac{1}{h^{2n+2}}\mathcal{E}^h&(u^h)  \\ = &\;
\frac{1}{8}\int_{\Omega^1} \mathcal{Q}_3\Big(\frac{1}{h^{n+1}}G(x',0)^{-1/2}\big((\nabla u^h)^{\trsp}\nabla
u^h(x', hx_3) - G(x',hx_3)\big)G(x',0)^{-1/2}\Big) \dd x\\ & +
\int_{\Omega^1} \frac{1}{h^{2n+2}}
\mathcal{O}\big(h|(\nabla u^h)^{\trsp}\nabla u^h - G|^2\big) +\frac{1}{h^{2n+2}}
 \mathcal{O}\big(|(\nabla u^h)^{\trsp}\nabla u^h - G|^3\big) \dd x.
\end{split}
\end{equation}

\medskip

\noindent We thus compute, for all $(x',x_3)\in \Omega^1$:
\[
\begin{aligned}
(\nabla u^h)^{\trsp} &\nabla u^h(x',hx_3) - G(x',hx_3) = 2 h^n
\big((\nabla y_0)^{\trsp}\nabla v^h\big)^*_{\sym} \\ & + \frac{h^{n+1}x_3^{n+1}}{(n+1)!} 
\bigg(\sum_{k=1}^n \binom{n+1}{k}B_k^{\trsp}B_{n+1-k} +
2\big(B_0^{\trsp}\big[\partial_1\vec b_{n+1},~\partial_2 \vec
b_{n+1},~\vec k_0\big]\big)_{\sym} - \partial_3^{(n+1)}G(x',0)\bigg)
\\ & + 2h^{n+1}x_3 \big(B_0^T\big[ \nabla \vec p^h,~ \vec r^h\big]\big)_{\sym} +  
2h^{n+1}x_3 \big(B_1^T\big[ \nabla \vec v^h,~ \vec p^h\big]\big)_{\sym} \\ & 
+ 2h^{n+1}\big(B_0^T\big[\nabla w^h,~ \vec q^h\big]\big)_{\sym} +\mathcal{R}_h,
\end{aligned}
\]
where:
$$\mathcal{R}^h = 
{o}(h^{n+1}) + \mathcal{O}(h^{n+2}) \big( |\nabla \vec v^h| +
|\nabla^2 \vec v^h|\big) + \mathcal{O}(h^{2n}) |\nabla^2 \vec v^h|^2 + 
\mathcal{O}(h^{2n+2}) |\nabla^2 \vec v^h|^2.$$

\medskip

{\bf 3.} We now estimate the two last (error) terms in the right hand
side of (\ref{caca}). Observe that:
\begin{equation*}
\begin{split}
|(\nabla u^h)^{\trsp}\nabla u^h- G| = & \; \mathcal{O}(h^{n+1})\big(1+
|\nabla v^h| + |\nabla w^h| + |\vec p^h|+ |\nabla \vec p^h| + |\vec
q^h| + |\vec r^h|\big) \\ & + \mathcal{R}^h + \mathcal{O}(h^n) |\big((\nabla
y_0)^{\trsp}\nabla v^h\big)_{\sym} |\\ = & \;\mathcal{O}(h^{n+1})\big(1+
|\nabla v^h| + |\nabla^2 v^h| + h^{-1/2}o(1)\big) + \mathcal{O}(h^n) |\big((\nabla
y_0)^{\trsp}\nabla v^h\big)_{\sym} | \\ & + 
\mathcal{O}(h^{2n})|\nabla v^h|^2 + \mathcal{O}(h^{2n+2}) |\nabla^2 v^h|^2,
\end{split}
\end{equation*}
where we have repeatedly used (\ref{ph}), (\ref{wh}), (\ref{qh}) and (\ref{rhh}), Consequently:
\begin{equation*}
\begin{split}
\frac{1}{h^{2n+2}}\mathcal{O}\big(|(\nabla u^h)^{\trsp}\nabla u^h-
G|^3\big) = & \; \mathcal{O}(h^{n+1})\big(1+
|\nabla v^h|^3 + |\nabla^2 v^h|^3 + h^{-3/2}o(1)\big) + 
\mathcal{O}(h^{4n+4})|\nabla^2 v^h|^6 \\ & + \mathcal{O}(h^{4n-2}) |\nabla
v^h|^6 + \mathcal{O}(h^{n-2}) |\big((\nabla y_0)^{\trsp}\nabla v^h\big)_{\sym} |^3. 
\end{split}
\end{equation*}
The first two terms in the right hand side above converge to $0$ in
$L^1(\omega^1)$ by (\ref{vh}) and (\ref{wh}). The $L^1$ norm of the
third term is bounded by $Ch^{4n-2}\|\nabla v^h\|_{W^{1,2}}^4$ and
thus converges to $0$ as well. The final fourth term is bounded,
in virtue of (\ref{vh}) by:
\begin{equation}\label{bdsym}
\begin{split}
\frac{1}{h^2}\int_\omega &\big|\big((\nabla y_0)^{\trsp}\nabla
v^h\big)_{\sym}\big|^2\dd x \leq \frac{C}{h^2} \big(\|\nabla v^h\|^2_{L^\infty}
+ \|\nabla^2 v^h\|^2_{L^\infty}\big) \int_{\{v^h\neq V\}} \dist^2(x',
\{v^h=V\})\dd x' \\ & \leq \frac{C\epsilon^2}{h^{2n+2}}\int_{\{v^h\neq V\}} \dist^2(x',
\{v^h=V\})\dd x' \leq \frac{C\epsilon^2}{h^{2n+2}}\big|\{v^h\neq
V\}\big|^2\\ & \leq \frac{C\epsilon^2}{h^{2n+2}}h^{4n} \cdot o(1) \to
0 \quad \mbox{ as } h\to 0.
\end{split}
\end{equation}
This completes the convergence analysis of the first error term in
(\ref{caca}). For the second term, we get:
\begin{equation*}
\begin{split}
\frac{1}{h^{2n+2}}\mathcal{O}\big(h|(\nabla u^h)^{\trsp}\nabla u^h-
G|^2\big) = & \; \mathcal{O}(h)\big(1+
|\nabla v^h|^2 + |\nabla^2 v^h|^2 + h^{-1}o(1)\big) + 
\mathcal{O}(h^{2n-1})|\nabla v^h|^4 \\ & + \mathcal{O}(h^{2n+3}) |\nabla^2
v^h|^4 + \frac{1}{h}\mathcal{O}\big(|\big((\nabla y_0)^{\trsp}\nabla v^h\big)_{\sym} |^2\big). 
\end{split}
\end{equation*}
As before, the first three terms converge to $0$ in $L^1(\omega)$,
whereas convergence of the last term follows by (\ref{bdsym}).
Concluding, and since $\frac{1}{h^{n+1}}\mathcal{R}^h$
converges to $0$ in $L^2(\Omega^1)$, the limit in (\ref{caca}) becomes:
\begin{equation}\label{caca2}
\begin{split}
\lim_{h\to 0}&\frac{1}{h^{2n+2}}\mathcal{E}^h(u^h)  \\ & = 
\lim_{h\to 0}\frac{1}{8}\int_{\Omega^1} \mathcal{Q}_3\Big(\frac{1}{h^{n+1}}G(x',0)^{-1/2}\big((\nabla u^h)^{\trsp}\nabla
u^h(x', hx_3) - G(x',hx_3)\big)G(x',0)^{-1/2}\Big) \dd x \\ & =
\lim_{h\to 0}\frac{1}{8}\int_{\Omega^1} \mathcal{Q}_3\Big(G(x',0)^{-1/2}K^h(x',x_3) G(x',0)^{-1/2}\Big) \dd x,
\end{split}
\end{equation}
where for a.e. $(x', x_3)\in\Omega^1$ we define:
\[
\begin{aligned}
K^h(x',x_3) = & \; \frac{2}{h} \big((\nabla y_0)^{\trsp}\nabla
v^h\big)^*_{\sym} \\ & + \frac{x_3^{n+1}}{(n+1)!} 
\bigg(\sum_{k=1}^n \binom{n+1}{k}B_k^{\trsp}B_{n+1-k} +
2\big(B_0^{\trsp}\big[\partial_1\vec b_{n+1},~\partial_2 \vec
b_{n+1},~\vec k_0\big]\big)_{\sym} - \partial_3^{(n+1)}G(x',0)\bigg)
\\ & + 2x_3 \big(B_0^T\big[ \nabla \vec p^h,~ \vec r^h\big] +  
\big(B_1^T\big[ \nabla \vec v^h,~ \vec p^h\big]\big)_{\sym} 
+ 2 \big(B_0^T\big[\nabla w^h,~ \vec q^h\big]\big)_{\sym}.
\end{aligned}
\]

\smallskip

\noindent In view of (\ref{bdsym}) and since $\|\vec r^h - \tilde
r^h\|_{L^2}$converges to $0$ as requested in (\ref{rhh}), the
compatibility in the definition (\ref{qh}) now yields from (\ref{caca2}):
\begin{equation}\label{caca3}
\begin{split}
\lim_{h\to 0}&\frac{1}{h^{2n+2}}\mathcal{E}^h(u^h)  \\ & =   
\lim_{h\to 0}\frac{1}{2}\int_{\omega} \mathcal{Q}_2\bigg(x', 
\frac{x_3^{n+1}}{2(n+1)!} 
\bigg(2\big((\nabla y_0)^{\trsp} \nabla \vec b_{n+1}\big)_{\sym } +
\sum_{k=1}^n \binom{n+1}{k}(\nabla\vec b_k)^{\trsp}\nabla \vec b_{n+1-k} 
\\ & \qquad \qquad\qquad\qquad\qquad \qquad\qquad\qquad \qquad\qquad \qquad\qquad\qquad
- \partial_3^{(n+1)}G(x',0)_{2\times 2}\bigg)
\\ & \quad \qquad\qquad\qquad\qquad +2x_3 \Big((\nabla
y_0)^{\trsp}\nabla\vec p^h 
+ (\nabla v^h)^{\trsp}\nabla \vec b_1\Big)
+  \big((\nabla y_0)^{\trsp}\nabla w^h\big)_{\sym} \bigg) \dd x'.
\end{split}
\end{equation}
Now, decomposing the integrand above as in the proof of Lemma
\ref{fin_strange} and recalling convergences in (\ref{vh}) and
(\ref{wh}), we conclude that the right hand side of (\ref{caca3}) equals 
$\mathcal{I}_{2(n+1)}(V)$, as claimed.
\end{proof}

\medskip

It is worth observing that directly from Theorems \ref{teo2} and \ref{teo3} we obtain:

\begin{cor}\label{26new}
Each functional $\mathcal{I}_{2(n+1)}$  attains its infimum and there holds:
$$\lim_{h\to 0} \frac{1}{h^{2(n+1)}}\inf\mathcal{E}^h =\min \;\mathcal{I}_{2(n+1)}.$$
The infima  in the left hand side are taken over 
$W^{1,2}(\Omega, \mathbb{R}^3)$ deformations $u^h$, whereas the minimum in
the right hand side is taken over admissible displacements $V\in
\mathcal{V}_{y_0}$.
\end{cor}


\begin{thebibliography}{11}

\bibitem{KoBe} {\sc P.  Bella and R.V.  Kohn}, {\em Metric-induced
    wrinkling of a thin elastic sheet}, J. Nonlinear Sci. {\bf 24}
  (2014), pp.~1147--1176.  

\bibitem{Bella} {\sc P. Bella and R.V. Kohn}, {\em The coarsening of folds in 
     hanging drapes}, Comm Pure Appl Math {\bf 70}(5) (2017), pp.~ 978--2012.

\bibitem{BCDM} {\sc H. Ben Belgacem, S. Conti, A. DeSimone and
    S. Muller}, {\em Rigorous bounds for the Foppl–von K\'arm\'an theory of
isotropically compressed plates}, J. Nonlinear Sci. {\bf 10},
(2000), pp.~661-- 683.

\bibitem{BCDM2} {\sc H. Ben Belgacem, S. Conti, A. DeSimone and
    S. Muller}, {\em Energy scaling of compressed elastic
    films—three-dimensional elasticity and reduced theories}, 
Arch. Ration. Mech. Anal. {\bf 164} (2002), no. 1, pp.~ 1--37.

\bibitem{Maggi} {\sc S. Conti and F. Maggi}, {\em Confining thin elastic sheets and folding paper}
Archive for Rational Mechanics and Analysis (2008), {\bf 187}, Issue 1, pp.~ 1--48.

\bibitem{BLS} {\sc K. Bhattacharya, M. Lewicka and M. Sch{\"a}ffner}, {\em Plates with
  incompatible prestrain}, Archive for Rational Mechanics and Analysis, 221
  (2016), pp.~143--181.

\bibitem{ESK1} {\sc E.  Efrati, E.  Sharon and R.  Kupferman}, {\em Elastic theory of
  unconstrained non-Euclidean plates}, J. Mech.  Phys. Solids, {\bf 57}
(2009), pp.~762--775. 

\bibitem{FJM02}
{\sc G. Friesecke, R. D.  James and S. M{\"u}ller}, {\em A theorem on geometric
  rigidity and the derivation of nonlinear plate theory from three-dimensional
  elasticity}, Comm. Pure Appl. Math., 55 (2002), pp.~1461--1506.

\bibitem{FJMhier} 
{\sc G. Friesecke, R. D. James and S. M\"uller}, {\em A hierarchy 
of plate models derived from nonlinear elasticity by gamma-convergence}, 
Arch. Ration. Mech. Anal.,  180  (2006),  no. 2, pp.~183--236.

\bibitem{Ge1} {\sc J. Gemmer and S. Venkataramani}, {\em
Shape selection in non-Euclidean plates}, Physica D: Nonlinear
Phenomena (2011), pp.~1536--1552.

\bibitem{Ge2}  {\sc J. Gemmer and S. Venkataramani}, {\em Shape
    transitions in hyperbolic non-Euclidean plates}, Soft Matter (2013), pp. ~8151--8161.

\bibitem{Ge3} {\sc J. Gemmer, E. Sharon, T. Shearman and
    S. Venkataramani}, {\em Isometric immersions, energy minimization
    and self-similar buckling in non-Euclidean elastic sheets},
  Europhysics Letters (2016), 24003.

\bibitem{biomi}
{\sc A. Gladman, E. Matsumoto, R. Nuzzo, L. Mahadevan and J. Lewis}, {\em Biomimetic 4D printing}, 
Nature Materials {\bf 15}, (2016) pp.~ 413--418.

\bibitem{JM} {\sc G. Jones and L. Mahadevan}, {\em Optimal control of plates using incompatible
strains}, Nonlinearity {\bf 28} (2015), 3153.

\bibitem{JS} {\sc W. Jin and P. Sternberg}, {\em Energy estimates for
    the von K\'arm\'an model of thin-film blistering}, Journal of
  Mathematical Physics {\bf 42}, 192 (2001).

\bibitem{10} {\sc R. Kempaiah and Z. Nie}, {\em From nature to
    synthetic systems: shape transformation in soft materials} J. Mater. Chem. B
  (2014) {\bf 2}, pp.~ 2357--2368.

\bibitem{9} {\sc J. Kim, J.  Hanna, M.  Byun, C. Santangelo, and R.
    Hayward}, {\em Designing responsive buckled surfaces by halftone
    gel lithography}, Science {\bf 335}, (2012) pp. ~1201--1205. 

\bibitem{sharon} {\sc Y.  Klein, E.  Efrati and E.  Sharon}, {\em Shaping of elastic sheets by prescription
of non-Euclidean metrics}, Science {\bf 315} (2007), pp.~1116--1120.

\bibitem{KM14} {\sc R. Kupferman and C.  Maor}, {\em A
    Riemannian approach to the membrane limit in non-Euclidean
    elasticity}, Comm. Contemp. Math. {\bf 16} (2014), no.5, 1350052.

\bibitem{KS14} {\sc R. Kupferman and J.P. Solomon}, {\em A
    Riemannian approach to reduced plate, shell, and rod theories},
  Journal of Functional Analysis {\bf 266} (2014), pp.~2989--3039.

\bibitem{12} {\sc M. Dias, J. Hanna and C. Santangelo}, {\em
    Programmed buckling by controlled lateral swelling in a thin
    elastic sheet},  Phys. Rev. E {\bf 84}, (2011), 036603.

\bibitem{LR1} {\sc H. Le Dret and A. Raoult}, {\em The nonlinear membrane
    model as a variational limit of nonlinear three-dimensional
    elasticity}, J. Math. Pures Appl. {\bf 73} (1995), pp.~549--578.

\bibitem{LR2} {\sc H. Le Dret and A. Raoult}, {\em The membrane shell
    model in nonlinear elasticity: a variational asymptotic
    derivation}, J. Nonlinear Sci.  {\bf 6} (1996), pp.~59--84.

\bibitem{lemapa2} {\sc M.  Lewicka, L.  Mahadevan and R. Pakzad}, {\em
Models for elastic shells with incompatible strains}, 
Proc. Roy. Soc. A {\bf 47 } (2014),  2165 20130604; pp.~1471--2946. 

\bibitem{lemapa2new} {\sc M.  Lewicka, L.  Mahadevan and R.  Pakzad},
{\em The Monge-Ampere constrained elastic theories of shallow
  shells}, Annales de l'Institut Henri Poincare (C) Non Linear
Analysis {\bf 34}, Issue 1, (2017), pp.~45--67. 

\bibitem{lemopa} {\sc M.  Lewicka, M. Mora and R.  Pakzad},
{\em The matching property of infinitesimal isometries on elliptic
  surfaces and elasticity of thin shells}, Arch. Rational
Mech. Anal. (3) {\bf 200} (2011), pp.~1023--1050. 

\bibitem{LOP} {\sc M.  Lewicka, P.  Ochoa and R.  Pakzad},
{\em Variational models for prestrained plates with Monge-Ampere
  constraint}, Diff. Integral Equations, {\bf 28}, no 9-10 (2015), pp.~861--898.

\bibitem{LP11}
{\sc M. Lewicka and R.  Pakzad}, {\em {Scaling laws for non-Euclidean plates
  and the $W^{2, 2}$ isometric immersions of Riemannian metrics}}, ESAIM:
  Control, Optimisation and Calculus of Variations, 17 (2011), pp.~1158--1173.

\bibitem{LRR}
{\sc M. Lewicka, A. Raoult and D. Ricciotti}, {\em Plates with
  incompatible prestrain of high order}, Annales de l'Institut Henri
Poincare (C)  Non Linear Analysis, 34, Issue (2017), pp.~ 1883--1912.

\bibitem{LL} {\sc M. Lewicka and D. Lucic}, {\em Dimension reduction for thin films with transversally varying
  prestrain: the oscillatory and the non-oscillatory case}, to appear (2018).

\bibitem{18b} {\sc H. Liang and L.  Mahadevan}, {\em The shape of a long
    leaf}, Proc. Nat. Acad. Sci. (2009).

\bibitem{MS18} {\sc C. Maor and A. Shachar}, {\em On the role of
curvature in the elastic energy of non-Euclidean thin bodies}, preprint.

\bibitem{Olber1} {\sc S. Muller and H. Olbermann}, {\em Conical
    singularities in thin elastic sheets}, 
Calculus of Variations and Partial Differential Equations,
(2014) {\bf 49}, Issue 3–4, pp.~ 1177--1186.

\bibitem{Olber2}  {\sc H. Olbermann}, {\em Energy scaling law for the
    regular cone}, Journal of Nonlinear Science
(2016) {\bf 26}, Issue 2, pp.~ 287--314.

\bibitem{Olber3}  {\sc H. Olbermann}, {\em On a boundary value problem for conically deformed
  thin elastic sheets}, preprint arXiv:1710.01707.

\bibitem{RHM} {\sc P.E.K. Rodriguez, A. Hoger and A. McCulloch}, {\em
    Stress-dependent finite growth in finite soft elatic tissues}
  J. Biomechanics {\bf 27} (1994), pp.~455--467. 

\bibitem{22a} {\sc E. Sharon, B. Roman and H.L. Swinney},  {\em Geometrically driven wrinkling
observed in free plastic sheets and leaves}, Phys. Rev. E {\bf 75}  (2007),
pp. ~046211--046217.

\bibitem{Tobasco} {\sc I. Tobasco}, {\em Curvature-driven wrinkiling
    of thin elastic shells}, preprint.

\bibitem{Venka} {\sc S. Venkataramani}, {\em Lower bounds for the
    energy in a crumpled elastic sheet—a minimal ridge}, Nonlinearity
{\bf 17} (2004), no. 1, pp.~301--312.

\bibitem{11} {\sc Z. Wei, J. Jia, J. Athas, C. Wang, S. Raghavan,
    T. Li and Z. Nie}, {\em Hybrid hydrogel sheets that undergo
    pre-programmed shape transformations}, Soft Matter {\bf 10},
  (2014), pp.~ 8157--8162.

\end{thebibliography}

\end{document}